\documentclass[a4paper,11pt]{article}
\usepackage{fullpage}
\usepackage{amsmath,amssymb,latexsym,amsthm}
\usepackage[pdftex]{graphicx}
\usepackage{amsmath,amssymb}
\usepackage{graphicx}
\usepackage{color, epsfig}
\usepackage{subfigure}
\definecolor{gold}{rgb}{0.85,.66,0}

\definecolor{purple}{rgb}{.403,.067,.53}

\def\card{\#}
\providecommand{\Chi}{\raisebox{0.45ex}{\( \chi \)}}
\providecommand{\R}{\mathbb{R}}

\providecommand{\N}{\mathbb{N}}

\newtheorem{theorem}{Theorem}
\newtheorem{definition}{Definition}
\newtheorem{corollary}{Corollary}
\newtheorem{proposition}{Proposition}
\newtheorem{lemma}{Lemma}
\newtheorem{remark}{Remark}
\def\Lip{\mbox{\rm Lip}}\title{Compactness estimates \\ for Hamilton-Jacobi equations depending on space}

\author{
Fabio Ancona\footnote{
Dipartimento di Matematica,
Universit\`a degli Studi di Padova,
Via Trieste 63, 35121 Padova, Italy
} ,
Piermarco Cannarsa\footnote{
Dipartimento di Matematica,
Universit\`a di Roma 'Tor Vergata', 
Via della Ricerca Scientica 1, 00133 Roma, Italy} ,
Khai T. Nguyen\footnote{
Department of Mathematics, Penn State University,
University Park, Pa. 16802, U.S.A.}
}

\date{\vspace{5pt}{\small (Dedicated to Prof. Tai Ping Liu in the occasion of his $70^{th}$ birthday)}
\\
\vspace{25pt}
March 31, 2015}
\begin{document}
\maketitle
\begin{abstract}
We study quantitative  estimates of compactness in $\mathbf{W}^{1,1}_{loc}$
for the map $S_t$, $t>0$
that associates to every given initial data $u_0\in \Lip (\mathbb{R}^N)$ the corresponding solution
$S_t u_0$ of  a  Hamilton-Jacobi equation
$$
u_t+H\big(x, \nabla_{\!x}  u\big)=0\,,
\qquad t\geq 0,\quad x\in \mathbb{R}^N,
$$
with a  convex and coercive Hamiltonian $H=H(x,p)$.
We provide upper and lower bounds of order $1/\varepsilon^N$
on the the Kolmogorov $\varepsilon$-entropy in $\mathbf{W}^{1,1}$ of the image through the map $S_t$
of sets of bounded, compactly supported initial data.
Quantitative estimates of compactness, as suggested  by P.D. Lax~\cite{Lax02},
could provide a measure of the order of ``resolution'' and of 
``complexity'' of a numerical method implemented for this equation.
We establish these estimates deriving accurate a-priori bounds
on the Lipschitz, semiconcavity and semiconvexity constant of a viscosity solution
when the initial data is semiconvex.
The derivation of a small time controllability result is also fundamental to
establish the lower bounds on the $\varepsilon$-entropy.
\end{abstract}
%
%
%
%
\section{Introduction}
The theory of viscosity solutions to first-order 
Hamilton-Jacobi equations of the form 
\begin{equation}
\label{HJ}
u_t(t,x)+H\big(x,\nabla_{\!x}  u(t,x)\big)=0\,,
\qquad t\geq 0,\quad x\in \mathbb{R}^N,
\end{equation}
provides existence, uniqueness and stability results.
The concept of viscosity solution was introduced
by  M.G. Crandall and P.-L. Lions  in~\cite{CL} 
to cope with the lack of classical (smooth) solutions of the Cauchy problem for~\eqref{HJ}
 globally defined in time. In fact, for such equations 
singularities in the gradient of the solution may arise in finite time, no matter how smooth the initial datum
\begin{equation}
\label{in-data}
u(0,\cdot)= u_0
\end{equation}
is assumed to be. 
We refer to~\cite{BCD} for a review of the notion of viscosity solution and the related theory for equation
of type~\eqref{HJ}.
In the case where the Hamiltonian $H(x,p)$
is smooth in both variables 
and convex in the $p$-variable, the viscosity solution $u(t,x)$ of~\eqref{HJ}-\eqref{in-data}, 
with initial datum $u_0:\mathbb{R}^N\to\mathbb{R}$
Lipschitz continuous, can be represented as the value function of a classical problem 
in the calculus of variation: 
\begin{equation}\label{Value-function}
u(t,x)= \min_{\xi\in AC([0,t],\mathbb{R}^N)}\bigg\lbrace{u_{0}(\xi(0))+\int_{0}^{t}L(\xi(s),\dot{\xi}(s))ds\ \Big|\ \xi(t)=x\ \bigg\rbrace},
\end{equation}
where $AC([0,t],\mathbb{R}^{N})$ is the class of absolutely continuous functions from $[0,t]$ to $\mathbb{R}^{n}$
and $L$ denotes the Legendre transform of  $H$ with respect to the second group of variables:
\begin{equation}\label{Legendre-L}
L(x,q)\ \doteq \ \sup_{p\in\mathbb{R}^N}\big\{p\cdot q-H(x,p)\big\}\qquad \forall~x,q\in\mathbb{R}^N.
\end{equation}
Under appropriate regularity assumptions 
on the map $x\mapsto H(x,p)$,
this fact implies that $u(t,x)$ is
locally semiconcave in $x$,
which in turn ensures that $u(t,\cdot)$ is almost everywhere
twice differentiable and that $\nabla_{\!x} u(t,\cdot)$ has locally bounded variation
($\nabla_{\!x} u(t,\cdot)\in \text{BV}_{loc}$),
i.e. that the distributional Hessian $D^2_x u(t,\cdot)$ is a symmetric matrix of Radon measures.

There is a vast literature concerning the structure and the regularity of the gradient 
of a viscosity solution to~\eqref{HJ}, see for example
\cite{BT, BDR, CF, CMS, CSo, Fl}.
Instead, in  this paper we are interested in analyzing the regularizing effect of the whole
semigroup map
$$
S_t: \Lip_{loc} (\mathbb{R}^N)\to \Lip_{loc}(\mathbb{R}^N), \quad  t>0
$$
that associates to every  initial data $u_0\in \Lip (\mathbb{R}^N)$
the unique viscosity solution $S_t u_0 \doteq u(t,\cdot)$ of the corresponding Cauchy problem~\eqref{HJ}-\eqref{in-data},
evaluated at time $t$.
Namely, 
for $\mathbf{W}^{1,\infty}$-bounded subsets $\mathcal{L}$ of $\Lip (\mathbb{R}^N)$ of the form
\begin{equation}
\label{CLM-1}
\mathcal{L}_{[R,M]}\doteq
\Big\lbrace{u_0\in \Lip(\mathbb{R}^N)\ \big|\ \mathrm{supp}(u_0)\subset [-R,R]^N\,,\; 
\Lip[u_0]\leqslant M\Big\rbrace},
\end{equation}
the semiconcavity constant of $S_t u_0$, $u_0\in \mathcal{L}$, on every  bounded subset $\Omega\subset \R^N$,
depends only on $\Omega$, $t$ and $\mathcal{L}$.
Hence, thanks to the local uniform semiconcavity of $S_t(\mathcal{L})$,
applying Helly's compactness theorem and a Poincar\'e inequality for BV-functions,
one can show that the image set $S_t(\mathcal{L})$ is compact  with respect to
the  $\mathbf{W}^{1,1}_{loc}$-topology. This property reflects the irreversibility
feature of the equation~\eqref{HJ} when the Hamiltonian $H(x,p)$ is convex in the 
$p$-variable. 
Here, we are concerned with the compactifying effect of the map $S_t$ when the space $\Lip(\mathbb{R}^N)$
is endowed with the $\mathbf{W}^{1,1}_{loc}$-topology, rather than the classical  $\mathbf{L}^\infty$-topology,
having in mind the $\mathbf{L}^1$-stability theory 
and the $\mathbf{L}^1$-error estimates 
established for approximate solutions of Hamilton-Jacobi equations~\cite{Lin-Tadmor},
which turn out to be sharper than  the $\mathbf{L}^\infty$ ones.

Inspired by a question posed by P.D. Lax~\cite{Lax02} within the context of conservation laws,
we employed in~\cite{APN} the concept of Kolmogorov $\varepsilon$-entropy to provide a quantitative estimate 
of this regularizing effect of the semigroup map in the case where the Hamiltonian $H=H(\nabla_{\!x}  u)$ is a convex
function depending only on the spatial gradient of the solution.
We recall the notion of $\varepsilon$-entropy introduced by A.~Kolmogorov~\cite{KT}:
\begin{definition}\label{DefKE}
Let $(X,d)$ be a metric space and let $K$ be a totally bounded subset of $X$. For $\varepsilon>0$, let $\mathcal{N}_{\varepsilon}(K|X)$  be the minimal number of sets in a cover of $K$ by subsets of $X$ having diameter no larger than $2\varepsilon$. Then the $\varepsilon$-entropy of $K$ is defined as
\begin{equation} \nonumber
\mathcal H_{\varepsilon}(K|X) \doteq \log_{2} \mathcal{N}_{\varepsilon}(K|X).
\end{equation}
Throughout the paper, we will call {\em $\varepsilon$-cover} a cover of $K$ by subsets of $X$ having diameter no larger than $2\varepsilon$.
\end{definition}

Actually, since in general $S_t u_0$, $u_0\in\mathcal{L}$,  is not an element of $\mathbf{W}^{1,1}(\R^N)$,
we have analyzed in~\cite{APN} the Kolmogorov entropy of the translated set $S_t(\mathcal{L})-S_t \,0$
which is a subset of $\mathbf{W}^{1,1}(\R^N)$.
The main result of the present paper extends the estimates
on the Kolmogorov entropy established in~\cite{APN}
to the semigroup map generated by~\eqref{HJ},
for Hamiltonians satisfying
the \textbf{Standing Assumptions}:
\begin{itemize}
%
\item[{\bf (H1)}] $H\in C^2(\mathbb{R}^N\times \mathbb{R}^N)$  is 
a coercive and convex
map
with respect to the  second group of variables, i.e. it satisfies
%
\begin{equation}
\label{H-coercive}
\lim_{|p|\rightarrow\infty}\frac{H(x,p)}{|p|}=+\infty
\qquad\ \forall~x\in\mathbb{R}^N,
\end{equation}
\begin{equation}
\label{convex-H}
%
D^2_p H(x,p)>0
\qquad\ \forall~x, p\in\mathbb{R}^N,
\end{equation}
where $D^2_p H(x,p)$ denotes the Hessian of $H$ with respect to the  $p$ variables
and the inequality is understood in the sense that $D^2_p H(x,p)$
is a positive definite matrix. 
\item[{\bf (H2)}] $H$ and its gradient 
satisfy the inequalities:
\begin{equation}
\label{subl-H2}
\begin{aligned}
H(x,p)&\geqslant -c_1\big(1+|x|\big)
\\
\noalign{\medskip}
\big\langle p, D_p H(x,p)\big\rangle - H(x,p) &\geqslant c_2\big|D_p H(x,p)\big|^\alpha -c_3
\\
\noalign{\medskip}
\big|D_x H(x,p)\big|&\leqslant c_4 \big|D_p H(x,p)\big|^\alpha +c_5
\end{aligned}
 \qquad\quad\forall~x,p\in\mathbb{R}^N\,,
\end{equation}
for some constants $c_1, c_3, c_4, c_5\geq 0$, $c_2>0$ and $\alpha>1$. 
\end{itemize}
%
In fact, we shall provide upper bounds on the Kolmogorov entropy 
of $S_t(\mathcal{L}_{[R,M]})-S_t \,0$ at any time $t>0$ and lower bounds
for times $t$ smaller that a quantity depending on $R, M$.
Specifically, we prove the following
\begin{theorem}
\label{upper-lower-estimate}
Let $H:\R^N\times\R^N\to \R$ be  a function satisfying  the assumptions~{\bf (H1)-(H2)}
and $\{S_t\}_{t\geqslant 0}$ be the semigroup of viscosity solutions generated by~\eqref{HJ}
on the domain $\Lip (\mathbb{R}^N)$. 
Then, given  $R,M>0$,  
letting $\mathcal{L}_{[R,M]}$ be the set defined in~\eqref{CLM-1}
the following hold.
\begin{enumerate}  
\item[(i)] 
For any $T>0$ and for every $\varepsilon>0$ sufficiently small,
one has
\begin{equation}
\label{Upper-est-H}
\mathcal{H}_{\epsilon}\Big(S_T(\mathcal{L}_{[R,M]})-S_T \,0\ \big|\ \mathbf{W}^{1,1}(\mathbb{R}^N)\Big)\leqslant \Gamma^+_{[R,M,N,T]}\cdot\frac{1}{\varepsilon^N},
\end{equation} 
with 
\begin{align}
\label{G+def}
\Gamma^+_{[R,M,N,T]}&\doteq 
 \omega_N^N \cdot \Big(4N\cdot\big(1+ \mu_T+(\kappa_{\scriptscriptstyle{T}}+1)l_T\big)\Big)^{\!4N^2},
\end{align}
where  $\omega_N$ denotes the Lebesgue measure of the unit ball of $\mathbb{R}^N$
and $l_T,\,\mu_T,\kappa_T$  are constants depending on $R,M,N,T$
defined in~\eqref{lt-def}, \eqref{muT-def}, \eqref{kappaT-def}. 
\item[(ii)] 
For any $0<T<\tau_{\!_{R,M}}$ and for every $\varepsilon>0$ sufficiently small,
one has
\begin{equation}
\label{Lower-est-H}
\mathcal{H}_{\epsilon}\Big(S_T(\mathcal{L}_{[R,M]})-S_T \,0\ \big|\ \mathbf{W}^{1,1}(\mathbb{R}^N)\Big)\geqslant \Gamma^-_{[R,M,N,T]}\cdot\frac{1}{\varepsilon^N},
\end{equation} 
with 
\begin{align}
\label{G-def}
\Gamma^-_{[R,M,N,T]}&\doteq 
\frac{1}{8\cdot\ln 2}\cdot \bigg(\frac{K_{\!_{R,M}}\, \omega_N\, r_{\!\!_R}^{N+1}}{48(N+1)\,2^{N+1}}\bigg)^{\!\!N}
\end{align}
where $\tau_{\!_{R,M}}, r_{\!\!_R}, K_{\!_{R,M}}$ are 
constants depending on $R,M,N$ defined as in Section~\ref{concl-proof-mainthm}.
\end{enumerate}
\end{theorem}

\noindent
Since the upper and lower bounds on the $\varepsilon$-entropy in ${\bf W}^{1,1}$ of $S_t(\mathcal{L}_{[R,M]})-S_t \,0$
are both of order $1/\varepsilon^N$, we deduce that, for Hamiltonians satisfying the assumptions {\bf (H1), (H2)},
such an $\varepsilon$-entropy is of the same
size $\approx 1/\varepsilon^N$ established in~\cite{APN}  
for Hamiltonians not depending on the space variable.
Entropy numbers play a central role in various areas of information theory and statistics
as well as of ergodic and learning theory. In the present setting, this concept  could 
provide a measure of the order of ``resolution'' and of the ``complexity'' of a numerical scheme, as suggested in~\cite{Lax78,Lax02}.
Roughly speaking, the order of magnitude of the $\varepsilon$-entropy
should indicate the minimum number of operations that one should perform 
in order to obtain an approximate solution with a precision of order $\varepsilon$
with respect to the considered topology.

\begin{remark}\rm
\label{assumpt-thm1}
Because of the assumption  {\bf (H1)}, for any given $x,q\in\R^N$ there exists some point $p_{x,q}$
where the supremum in~\eqref{Legendre-L} is attained (cfr.~\cite[Appendix A.2]{CS}).
Thus, in particular, one finds that 
\begin{equation}
\label{L1-id}
L(x,0)=-H(x,p_{x}),
\end{equation}
for some $p_x\in\R^n$,
and
\begin{equation}
\label{L2-id}
\begin{aligned}
q&= D_p H(x,p_{x,q}),
\\
\noalign{\smallskip}
L(x,q)&=\langle p_{x,q}, q\rangle - H(x,p_{x,q}),
\\
\noalign{\smallskip}
D_x L(x,q)&= - D_x H(x,p_{x,q}),
\end{aligned}
\end{equation}
%
for some $p_{x,q}\in \R^N$.
Hence, relying on the inequalities of the assumption  {\bf (H2)}, one can show that
\begin{align*}
L(x,0)&\leqslant c_1(1+|x|)\qquad\forall~x\in\mathbb{R}^N\,,
\\
\noalign{\medskip}
L(x,q)&\geqslant c_2 |q|^\alpha - c_3\qquad\forall~x,q\in\mathbb{R}^N\,,
\\
\noalign{\medskip}
\big|D_x L(x,q)\big|&\leqslant c_4 |q|^\alpha +c_5
\qquad\forall~x,q\in\mathbb{R}^N\,.
\end{align*}
%
These uniform bounds on the Legendre transform of $H$ are fundamental to provide an estimate on the size of the support of 
the map $x\mapsto S_t u_0(x)- S_t \,0\, (x)$, when $u_0$ varies in 
a set $\mathcal{L}_{[R,M]}$ as in~\eqref{CLM-1}, as well as to derive a-priori bounds on the
minimizers for~\eqref{Value-function}.
The assumptions {\bf (H1)}-{\bf (H2)} are verified by a large class of Hamiltonians $H(x,p)$
convex in the $p$-variable. For example, if we consider 
\begin{equation}
\nonumber
H(x,p)=f(x)\big(1+|p|^2\big)^{m} + g(x),
\end{equation}
where $m> 1/2$ 
and $f,g \in C^2(\mathbb{R}^N)$ are such that
\begin{equation}
\nonumber
0< f(x)\leqslant c_f,
\qquad - c_g(1+|x|)\leqslant g(x)\leqslant c_g \qquad\forall~x\in\mathbb{R}^N\,,
\end{equation}
for some $c_f, c_g>0$, it is straightforward to verify that $H$ satisfies  {\bf (H1)}-{\bf (H2)}
for $\alpha=\frac{2m}{2m-1}$.
On the other hand, one can easily check that the assumption {\bf (H2)} is  certainly fulfilled,
in particular, by the Hamiltonians that satisfy {\bf (H1)} together with the (stronger) bounds
%
%
%
\begin{equation}
\label{subl-H2prime}
\begin{aligned}
-c'_{1}(&1+|x|)\leqslant H(x,0)\leqslant c'_1
\\
\noalign{\medskip}
&\big|D_p H(x,p)\big|^\alpha\leqslant c'_2 (1+|p|^2)
\end{aligned}
 \qquad\quad\forall~x,p\in\mathbb{R}^N\,,
\end{equation}
for some constants $c'_1, c'_2>0$ and $\alpha>1$.
\end{remark}
%

The key step of the proof of Theorem~\ref{upper-lower-estimate}-(i) consists in deriving accurate estimates 
on the size of the support $\omega_T(u_0)\doteq  \mathrm{Supp}(S_T u_0-S_T \,0)$
of the map $x\mapsto S_T u_0(x)- S_T \,0\, (x)$
and on the semiconcavity constant of $S_T u_0$ on $\omega_T(u_0)$,
when $u_0\in \mathcal{L}_{[R,M]}$. Notice that,
since the Hamiltonian depends on the space variable, we cannot employ the explicit
Hopf-Lax representation formula for the solutions of an Hamilton-Jacobi equation as in~\cite{APN}.
Instead, we shall obtain these estimates relying on the representation~\eqref{Value-function} 
of a solution to~\eqref{HJ} as as the value function of a Bolza problem in the calculus of variations 
and performing a careful analysis of the behaviour of the solution along the corresponding minimizers.
Thanks to such a-priori bounds, one then recovers the upper estimate~\eqref{Upper-est-H}
invoking a similar estimate established in~\cite{APN} for the Kolmogorov entropy of a class $\mathcal{SC}_{[K]}$
of semiconcave functions with semiconcavity constant $K$ defined 
on a bounded domain.
\smallskip

The proof of Theorem~\ref{upper-lower-estimate}-(ii), as in~\cite{APN}, is based on a controllability type
result for the elements of the class $\mathcal{SC}_{[K]}$. Namely, we show that, 
for times $T>0$ sufficiently small
and for some constant $K$ depending on $R,M$,
every element of $\mathcal{SC}_{[K]}$
which coincides with $S_T \,0$ outside a bounded domain
can be obtained
as the value $u(T,\cdot)$ of a viscosity solution of~\eqref{HJ} with initial data in 
$\mathcal{L}_{[R,M]}$.
Notice that $S_T \,0$ is in general not a smooth function. Therefore, to establish such
a controllability property one cannot expect to produce smooth solutions on the whole domain $[0,T]\times \R^N$
that attain at time $T$ the desired profile.
However, we shall achieve this result relying on a fine analysis of the backward and forward
minimizers of a local smooth solution of~\eqref{HJ}
and performing accurate estimates on the semiconcavity and semiconvexity costants
of a viscosity solution of~\eqref{HJ}. 
In turn, this result yields the lower bound~\eqref{Lower-est-H} invoking the same type of estimates
provided in~\cite{APN} for the Kolmogorov entropy of  $\mathcal{SC}_{[K]}$.
It remains open the question wether a global (in time) controllability property for
semiconcave functions hold for Hamilton-Jacobi equations with Hamitonian depending
on space (cfr. remark~\ref{control-prop1} in Section~\ref{concl-proof-mainthm}).

\medskip

The paper is organized as follows. In Section~\ref{sec:prelim}, we collect  preliminary results and definitions
concerning semiconcave functions and Hamilton-Jacobi equations, as well as the quantitative
compactness estimates
on classes of semiconcave functions established in~\cite{APN}. 
In Section~\ref{sec:up-ext} we derive local a-priori bounds on the Lipschitz and
semiconcavity constant of a viscosity solution to~\eqref{HJ}, which then yield 
the upper bound stated in Theorem~\ref{upper-lower-estimate}-$(i)$. 
In the first part of Section~\ref{sec:low-est} we  provide local a-priori bounds on the 
semiconvexity constant of a viscosity solution to~\eqref{HJ} when
the initial data is semiconvex.
Next, we establish a local controllability result for a class of semiconcave
functions, which allows us to obtain
the lower bound stated in Theorem~\ref{upper-lower-estimate}-$(ii)$.
\medskip

\section{Notation and preliminaries}
\label{sec:prelim}

Let $N\geqslant 1$ be an integer. Throughout the paper we shall denote by:
\begin{itemize}
\item $|\cdot|$ the Euclidean norm in $\R^N$,
\item $\langle\cdot,\cdot\rangle$ the Euclidean inner product in $\R^N$,  
\item $[x,y]$ the segment joining two points $x,y\in \mathbb{R}^N$,
\item $B(x_0,r)$  the  open ball of $\R^N$ with radius $r >0$ and centered at $x_0$,
\item $\card (S)$ the number of elements of any finite set $S$,
\item $\mathrm{Vol}(D)$ the Lebesgue measure of a measurable set $D\subset \R^N$,
\item $\omega_N :=\mathrm{Vol}(B(0,1))=\frac{\pi^{N/2}}{\Gamma(N/2+1)}$  the Lebesgue measure of the unit ball of $\mathbb{R}^N$,
\item $\|A\|$ the usual operator norm of the $N\times N$ matrix $A$,
\item $\Lip (\R^N)$ the space of all Lipschitz continuous functions $f:\R^N\to \R$, and by $\Lip[f]$ the Lipschitz seminorm of $f$,
while $\Lip[f; \Omega]$ denotes the Lipschitz seminorm of the restriction of $f$ to a domain $\Omega\subset\R^N$,
\item $\mathrm{supp}(u)$ the support of $u\in\Lip (\R^N)$, that is, the closure of $\big\{x\in\R^N~|~u(x)\neq 0\big\}$,
\item $AC([a,b],\mathbb{R}^{N})$, with $[a,b]$ interval of $\R$, the space of all absolutely continuous functions from $[a,b]$ to~$\mathbb{R}^{N}$,
\item $\mathbf{L}^{1}(D)$, with $D\subset\R^N$  a  measurable set, the Lebesgue space of all (equivalence classes of) summable functions on $D$, equipped with the usual norm $\|\cdot\|_{\mathbf{L}^{1}(D)}$, 
\item $\mathbf{L}^{\infty}(D)$, with $D\subset\R^N$  a  measurable set,  the space of all essentially bounded functions on~$D$,  and by $\|u\|_{\mathbf{L}^{\infty}(D)}$ the essential supremum of a function $u\in \mathbf{L}^{\infty}(D)$ (we shall use the same symbol in case $u$ is vector-valued), 
\item $\mathbf{W}^{1,1}\big(\Omega)$, with $\Omega$ a convex domain in $\R^N$, 
the  Sobolev space of functions with summable first order distributional derivatives, and by $\|\cdot\|_{\mathbf{W}^{1,1}(\Omega)}$ its norm,
\item $\mathbf{W}^{1,1}_0\big(\Omega)$, with $\Omega$ a convex domain in $\R^N$, 
the  Sobolev space of functions $F\in \mathbf{W}^{1,1}\big(\Omega)$ with zero trace on the boundary $\partial \Omega$,
\item $BV(\Omega,\R^N)$, with $\Omega$ a domain in $\R^N$, the space of all vector-valued functions $F:\Omega\to\R^N$ of bounded variation (that is, all $F\in \mathbf{L}^1(\Omega,\R^N)$ such that the first partial derivatives of $F$ in the sense of distributions are  measures with finite total variation in $\Omega$).
\end{itemize}
Moreover $\lfloor a\rfloor:=\max\{z\in \mathbb{Z}\, | x\leq a\}$  denotes the integer part of $a$.
\medskip

\subsection{Generalized gradients and semiconcave functions}
\label{subsec:semiconc}
 We shall adopt the notation $Du$
for the distributional gradient of a continuous function~$u$.
A notion of generalized differentials that specially fits  viscosity solutions is recalled in the following
\begin{definition}
\label{general-grad}
Let $u:\Omega\rightarrow\mathbb{R}$, with $\Omega\subseteq\mathbb{R}^N$ open. For every $x\in\Omega$, the sets
\begin{equation}
\label{sup-sub-diff}
\begin{aligned}
D^+u(x)
&:=\bigg\lbrace{p\in\mathbb{R}^N\ |\ \limsup_{y\rightarrow x}\frac{u(y)-u(x)-\langle p,y-x\rangle}{|y-x|}\leqslant 0 \bigg\rbrace},
\\
\noalign{\medskip}
D^-u(x)
&:=\bigg\lbrace{p\in\mathbb{R}^N\ |\ \liminf_{y\rightarrow x}\frac{u(y)-u(x)-\langle p,y-x\rangle}{|y-x|}\geqslant 0 \bigg\rbrace},
\end{aligned}
\end{equation}
are called, respectively, the $D$-\textit{superdifferential} 
and the $D$-\textit{subdifferential} 
of $u$ at $x$. Moreover, 
\begin{equation}
\label{reach-grad}
D^*u(x):=\Big\lbrace{p=\lim_{k\rightarrow\infty}\nabla u(x_k)\ |\ f\ \mathrm{is\ differentiable\ at}\ x_k\ \mathrm{and}\ x_k\rightarrow x\Big\rbrace},
\end{equation}
is called the set of reachable gradients of $u$ at $x$.
\end{definition}
\noindent
From  definition~\eqref{sup-sub-diff} it follows that there holds
\begin{equation}
\label{sup-sub-diff-equal}
D^-u(x)=-D^+(-u)(x)\qquad\quad\forall~x\in\Omega.
\end{equation}
\begin{remark}
\label{clarke-grad}
\rm
\label{semiconc-prop}
When $u$ is locally Lipschitz in $\Omega$, $D^*u(x)$ is a nonempty compact set for every $x\in\Omega$. Moreover, if $L$ is a Lipschitz constant for $u$ on a neighborhood of $x$, then we have that
\begin{equation*}
|p|\leqslant L\qquad \forall\,p\in D^*u(x).
\end{equation*}
In this case, the convex hull co$\,D^*u(x)$ gives Clarke's generalized gradient, $\partial u(x)$, see \cite{C}. Consequently, there also holds
\begin{equation}
\label{eq:gd}
|p|\leqslant L\qquad \forall\,p\in \partial u(x).
\end{equation}
On the other hand, if $u$ is semiconvex then one has $\partial u(x)=D^-u(x)$.
 \end{remark}

We collect below some basic definitions and properties of semiconcave  functions in~$\R^N$ that will be used 
in the paper. We refer the reader to~\cite{CS} for a comprehensive introduction to the corresponding theory.
\begin{definition}\label{SCL} 
A continuous function $u:\Omega\rightarrow \mathbb{R}$, with $\Omega\subset\mathbb{R}^N$,
 is called {\em semiconcave} if there exists $K>0$ such that 
\begin{equation}\label{eq:semiconcave}
u(x+h)+u(x-h)-2u(x)\leqslant K|h|^2,
\end{equation}
for all $x,h\in\R^N$ such that $[x-h,x+h]\subset \Omega$. 
When this property holds true, we also say that $u$ is {\em semiconcave in $\Omega$ with constant}  $K$, and call $K$ a {\em semiconcavity constant} for~$u$.  
\begin{itemize}
\item[-]We say that $u$ is {\em semiconvex}  with constant $K$ if $-u$ is semiconcave with constant $K$.
\item[-]We say that $u:\Omega\rightarrow \mathbb{R}$, with $\Omega\subset\mathbb{R}^N$ open, 
is {\em locally semiconcave} (or locally semiconvex) if 
$u$ is semiconcave (semiconvex) in every compact set $A \subset\subset\Omega$.
\end{itemize}
\end{definition}
\begin{remark}
\rm
\label{semiconc-prop-2}
The notion of semiconcavity introduced here is the most commonly used in the literature,
often denoted as {\it linear semiconcavity}.
A more general definition of semiconcavity can be found in~\cite{CS}.
 It is easy to see that a function $u$  is semiconcave (semiconvex) in $\Omega$ with  constant $K>0$ if any only if the function 
 $$\widetilde{u}(x)\doteq u(x)-\frac{K}{2}|x|^2\qquad \bigg(\widetilde{u}(x)\doteq u(x)+\frac{K}{2}|x|^2\bigg),  \qquad x\in\Omega$$ 
 is concave (convex).
 \end{remark}

Semiconcave functions and their
superdifferential enjoy the properties stated 
in the following (see~\cite[Theorem 2.31, Proposition~3.3.1, Proposition~3.3.4, Theorem~3.3.6, Proposition~3.3.10]{CS})
\begin{theorem}
\label{Pro-Semi}
Let $\Omega\subseteq\mathbb{R}^N$ be open
and $u:\Omega\rightarrow \mathbb{R}$ be semiconcave with semiconcavity constant~$K$. Then, the following properties hold true.
\begin{enumerate}
\item [(i)] u is Lipschitz continuous and almost everywhere 
differentiable.
\item [(ii)] The superdifferential $D^{+}u(x)$ is a compact, convex,
nonempty set for  all $x\in\Omega$.
\item [(iii)] $D^{+}u(x)=\mathrm{co}\,D^{*}u(x)$ for  all $x\in\Omega$, where {\em co} stands for the convex hull.
\item [(iv)] $D^{+}u(x)$ is a singleton if and only if $u$ is differentiable at $x$.
\item[(v)] If $D^{+}u(x)$ is a singleton for every $x\in\Omega$, 
then $u\in C^{1}(\Omega,\mathbb{R})$. 
\item[(vi)] $p\in D^{+}u(x)$  if and only if
\begin{equation}\label{eq:semiconcave-grad}
u(x+h)-u(x)-\langle p, h\rangle\leqslant \frac{K}{2}|h|^2,
\end{equation}
for all $x,h\in\R^N$ such that $[x,x+h]\subset \Omega$. 
\end{enumerate}
\end{theorem}
\medskip
\begin{proposition}
\label{prop2}
Let $\Omega\subseteq\mathbb{R}^N$ be open convex
 and $u:\Omega\rightarrow \mathbb{R}$ be semiconcave with  constant $K$. 
Then, 
for every $x,y\in\Omega$, 
and for any $p_x\in D^+u(x)$, $p_y\in D^+u(y)$, there holds
\[
\big\langle p_y-p_x,y-x\big\rangle\leqslant K\, |y-x|^2.
\]
\end{proposition}
\begin{remark}\rm
\label{semiconc-semiconc}
Relying on the properties of the generalized gradients one can show that if a 
function $u:\Omega\rightarrow\mathbb{R}$ ( $\Omega\subseteq\mathbb{R}^N$ open and convex)
is both semiconcave and semiconvex in $\Omega$
then $u\in C^{1,1}(\Omega,\mathbb{R})$ (see~\cite[Corollary 3.3.8]{CS}).
\end{remark}

\subsection{Upper and lower bounds on the $\varepsilon$-entropy for a class of semiconcave functions}
We report here the estimates obtained in~\cite{APN} on the $\varepsilon$-entropy in ${\bf W^{1,1}}$ of classes of semiconcave
functions on $\mathbb{R}^N$.

Given any $R,M,K>0$, and any $\psi \in \Lip([-R, R]^N)$,
consider the 
classes of functions
\begin{equation}
\label{LRM}
\begin{aligned}
\mathcal L_{[R,M]}^{\psi}
&\doteq \Big\lbrace{u_0\in \Lip([-R, R]^N)\ \big|\ u_0=\psi \ \ \text{on} \ \ \partial\, [-R, R]^N\,,\; 
\Lip\big[u_0]
\leqslant M\Big\rbrace},
\end{aligned}
\end{equation}
and 
\begin{equation}
\label{LSRM}
\begin{aligned}
\mathcal{SC}_{[R,M,K]}^{\psi}
&\doteq \Big\lbrace{
u_0\in \mathcal L_{[R,M]}^{\psi}\, \big| \, u_0
\ \text{is semiconcave on} \ [-R,R]^N \ \text{with semiconcavity constant} \ K
\Big\rbrace}.
\end{aligned}
\end{equation}
%

%
\begin{theorem}
\label{estSC}
Given any $R,M,K>0$ and a semiconcave function $\psi \in \Lip([-R,R]^N)$
having Lipschitz constant $M$ and semiconcavity constant $K$, with the above notations
the followings hold:
\begin{itemize}
\item[(i)]
for every $0<\varepsilon< \frac{M R^2}{5}(\min\{R,M\})^N$, one has
\begin{equation}
\label{upper-entr-sc-1}
 \mathcal{H}_{\varepsilon}\Big(\mathcal{SC}^{\psi}_{[R,M,K]}
 \ \big|\ \mathbf{W}^{1,1} \big([-R,R]^N\big)\Big)\leqslant
\gamma^+_{_{[R,M,K,N]}}\cdot\frac{1}{\varepsilon^N}
\end{equation}
where
 \begin{equation}\label{Consta1}
\gamma^+_{_{[R,M,K,N]}}\doteq
 \omega_N^N \cdot \Big(4N\cdot\big(1+ M+(K+1)R\big)\Big)^{\!4N^2}\,.
 \end{equation}
\item[(ii)] 
for every $0<\varepsilon<\frac{\omega_N R^N}{N\,2^{N+9}}\min\{M,K\}$, one has
\begin{equation}
\label{low-entr-sc}
\mathcal{H}_{\varepsilon}\Big(\mathcal{SC}^{\psi}_{[R,2M,2K]}\ \big|\ \mathbf{W}^{1,1}\big([-R,R]^N\big)\Big)\geqslant
\gamma^-_{_{[R,K,N]}}\cdot\frac{1}{\varepsilon^N},
\end{equation}
where 
\begin{equation}\label{eq:Gamma}
\gamma^-_{_{[R,K,N]}}\doteq \frac{1}{8\cdot\ln 2}\cdot \bigg(\frac{K\, \omega_N\, R^{N+1}}{48(N+1)\,2^{N+1}}\bigg)^{\!\!N}.
\end{equation}
 \end{itemize}
\end{theorem}
%
%
\begin{proof}
The estimates stated in Theorem~\ref{estSC} were established in~\cite[Proposition~8, Proposition~10]{APN} for the class of functions
\begin{equation}
\label{SC-old}
\mathcal{SC}_{[R,M,K]}
\doteq \Big\lbrace{
u_0\in \mathcal L_{[R,M]}\, \big| \, u_0 \ \text{is semiconcave with semiconcavity constant} \ K
\Big\rbrace}
\end{equation}
(with $ \mathcal L_{[R,M]}$ as in~\eqref{CLM-1}), which consists of the extensions to $\R^N$ of the elements 
in $\mathcal{SC}^0_{[R,M,K]}$.
However, with the same arguments of the proof of~\cite[Proposition~8]{APN} one  obtains the upper bound
\eqref{upper-entr-sc-1} for the class of function in~\eqref{LRM}
and for a general $\psi \in \Lip([-R,R]^N)$.
In fact, as in~\cite{APN}, we may define the map  $\mathcal{T}_K:  \mathbf{L}^1([-R,R]^N) \to  \mathbf{L}^1([-R,R]^N)$,
that associates to any $f\in \mathbf{L}^1([-R,R]^N)$ the function
\begin{equation}
\label{Tk-map}
\mathcal{T}_K f(x) \doteq  
f(x)+\frac{K}{2}|x|^2\,,
%
\end{equation}
and then consider the class of concave functions 
\begin{equation}
\label{C-klm-class}
\mathcal{C}^\psi_{[R,M,K]} \doteq \Big\lbrace{f\in\mathbf{W}^{1,1}([-R,R]^N) \ \big{|}\ \mathcal{T}_K f\in\mathcal{SC}^\psi_{[R,M,K]} \Big\rbrace}.
\end{equation}
Since $\mathcal{T}_K: \mathcal{C}^\psi_{[R,M,K]}\to \mathcal{SC}^\psi_{[R,M,K]}$ is a surjective isometry,
it is sufficient to provide an upper bound on the $\varepsilon$-entropy of the set $\mathcal{C}^\psi_{[R,M,K]}$.
Then, observing that for every $f_1, f_2\in \mathcal{C}^\psi_{[R,M,K]}$ there holds
$f_1-f_2\in \mathbf{W}^{1,1}_0([-R,R]^N)$,
and applying  the Poincar\'e inequality
for trace-zero $\mathbf{W}^{1,1}$ functions, 
we produce as shown in~\cite{APN}
an $\varepsilon$-covering of  $\mathcal{C}^\psi_{[R,M,K]}$ in $\mathbf{W}^{1,1}$
with a cardinality of order $2^{\frac{\gamma^+_{_{[R,M,K,N]}}}{\varepsilon^N}}$.
This yields the upper bound~\eqref{upper-entr-sc-1}.
Similarly, one can recover the lower bound~\eqref{low-entr-sc}
as follows.
For $\varepsilon$ sufficiently small, it was shown in
the proof of~\cite[Proposition~10]{APN}
that there exists a class of $C^1$ semiconcave functions $\mathcal{U}_\varepsilon\subset \mathcal{SC}_{[R,M,K]}$
for which
\begin{equation}
\label{low-entr-sc-old}
\mathcal{H}_{\varepsilon}\Big(\mathcal{U}_\varepsilon\ \big|\ \mathbf{W}^{1,1}\big(\R^N\big)\Big)\geqslant
\gamma^-_{_{[R,K,N]}}\cdot\frac{1}{\varepsilon^N}.
\end{equation}
On the other hand, since $\psi$ is a semiconcave function with Lipschitz constant $M$
and semiconcavity constant $K$, by definition~\eqref{LSRM}
it follows that the restrictions to $[-R, R]^N$ of the maps in $\mathcal{U}_\varepsilon+\psi$
are all elements of $\mathcal{SC}^\psi_{[R,2M,2K]}$.
Thus, one has
\begin{equation}
\begin{aligned}
\mathcal{H}_{\varepsilon}\Big(\mathcal{SC}^{\psi}_{[R,2M,2K]}\ \big|\ \mathbf{W}^{1,1}\big([-R,R]^N\big)\Big)
&\geqslant 
\mathcal{H}_{\varepsilon}\Big(\mathcal{U}_\varepsilon+\psi\ \big|\ \mathbf{W}^{1,1}\big([-R,R]^N\big)\Big)
\\
&=\mathcal{H}_{\varepsilon}\Big(\mathcal{U}_\varepsilon\ \big|\ \mathbf{W}^{1,1}\big([-R,R]^N\big)\Big)
\\
&=\mathcal{H}_{\varepsilon}\Big(\mathcal{U}_\varepsilon\ \big|\ \mathbf{W}^{1,1}\big(\R^N\big)\Big),
\end{aligned}
\end{equation}
which, together with~\eqref{low-entr-sc-old}, yields~\eqref{low-entr-sc}.
\end{proof}
\begin{remark}
\label{SC-low-entr}\rm
The same lower bound~\eqref{low-entr-sc}  of Theorem~\ref{estSC}-(ii) holds 
for the class of $C^1$ elements of the set $\mathcal{SC}_{[R,M,K]}^{\psi}$:
%
\begin{equation}
\label{LSRM-1}
\mathcal{SC}_{[R,M,K]}^{\psi,1}
\doteq \Big\lbrace
u_0\in \mathcal{SC}_{[R,M,K]}^{\psi}\cap C^1([-R,R]^N)
\ \big| \,  D u_0 = D \psi\ \ \text{on} \ \ \partial\, [-R, R]^N
\Big\rbrace.
\end{equation}
\end{remark}
\medskip

\subsection{Hamilton Jacobi equation}
\label{subsec:HJ}
Consider the Hamilton-Jacobi equation~\eqref{HJ}, and observe
that 
the assumptions {\bf (H1)-(H2)} imply that the Legendre transform $L(x,q)$ of $p\to H(x,p)$ 
defined in~\eqref{Legendre-L} enjoy similar properties as $H$
(cfr. Remark~\ref{assumpt-thm1} in the Introduction and \cite[Appendix~A.2]{CS}):
\begin{itemize}
%
\item[{\bf (L1)}] $L\in C^2(\mathbb{R}^N\times \mathbb{R}^N)$ is a 
convex and 
coercive map with respect to the second group of variables, i.e., 
\begin{equation}
\label{unif-semiconc-L}
0<D^2_q L(x,q)
\qquad\ \forall~x, q\in\mathbb{R}^N,
\end{equation}
\begin{equation}
\label{coercive-L}
\lim_{|q|\rightarrow\infty}\frac{L(x,q)}{|q|}=+\infty
\qquad \forall~x\in\R^N.
\end{equation}
%
\item[{\bf (L2)}] There exist constants $c_1, c_3, c_4, c_5\geq 0$, $c_2>0$ and $\alpha>1$
so that
\begin{align}
\label{subl-L1}
L(x,0)&\leqslant c_1(1+|x|)\qquad\forall~x\in\mathbb{R}^N\,,
\\
\noalign{\medskip}
\label{subl-L2}
L(x,q)&\geqslant c_2 |q|^\alpha - c_3\qquad\forall~x,q\in\mathbb{R}^N\,,
\\
\noalign{\medskip}
\label{subl-L3}
\big|D_x L(x,q)\big|&\leqslant c_4 |q|^\alpha +c_5
\qquad\forall~x,q\in\mathbb{R}^N\,.
\end{align}
\end{itemize}
%

%
%

As we mentioned in the introduction, since 
solutions of the Cauchy problem for~\eqref{HJ} may develop singularities in the gradient
in finite time, even with smooth initial data, a concept of generalized solution,
the {\it viscosity solution}, 
was  introduced in~\cite{CL} (see also \cite{CEL}).
We recall here the:
\begin{definition}\label{viscosity-solution}
We say that a continuous function $u:[0,T]\times\mathbb{R}^N$ is a viscosity solution of \eqref{HJ} if:
\begin{enumerate}
\item[$\mathrm{(i)}$] u is a viscosity subsolution of \eqref{HJ}, i.e., for every point $(t_0,x_0)\in \,]0,T[\,\times\mathbb{R}^N$ and test function $v\in C^1\big((0,+\infty)\times\mathbb{R}^N\big)$ such that $u-v$ has a local maximum at $(t_0,x_0)$, it holds
\[
v_t(t_0,x_0)+H\big(x_0,\nabla_{\!x} v(t_0,x_0)\big)\leqslant 0\,,
\] 
\item[$\mathrm{(ii)}$]  u is a viscosity supersolution of (\ref{HJ}), i.e., for every point $(t_0,x_0)\in \,]0,T[\,\times\mathbb{R}^N$ and test function $v\in C^1\big((0,+\infty)\times\mathbb{R}^N\big)$ such that $u-v$ has a local minimum at $(t_0,x_0)$, it holds
\[
v_t(t_0,x_0)+H\big(x_0,\nabla_{\!x} v(t_0,x_0)\big)\geqslant 0\,.
\] 
\end{enumerate}
In addition, we say that $u$ is a viscosity solution of the Cauchy problem~\eqref{HJ}-\eqref{in-data} 
if condition~\eqref{in-data} is satisfied in the classical sense.
\end{definition}
\begin{remark}
\label{visos-sol-differentiab}\rm
An alternative equivalent definition of viscosity solution is expressed in terms of 
the sub and superdifferential of the function (see~\cite{CEL}).
Relying on this definition,
and because  of Theorem~\ref{Pro-Semi}-$(iv)$,
one immediately see that every $C^1$ solution of~\eqref{HJ} is also a viscosity solution
of~\eqref{HJ}. On the other hand, if $u$ is a viscosity solution of~\eqref{HJ}, then $u$ satisfies the equation
at every point of differentiability. 
Moreover, one can show that if  $u:[0,T]\times \Omega$, $\Omega\subset\R^N$, is a viscosity solution of~\eqref{HJ} 
and we know that 
$u(t,\cdot )$ is both semiconcave and semiconvex in $\Omega$ for all $t\in \,]0,T]$,
then $u$ is a continuously differentiable 
classical solution of~\eqref{HJ} on~$]0,T]\times \Omega$
(e.g. see~\cite[Proposition~3]{APN}).
\end{remark}

Under the assumption \textbf{(L1)}, 
for every $u_{0}\in \Lip(\mathbb{R}^N)$,
the value function defined in~\eqref{Value-function}
in connection with the Bolza problem of calculus of variation with running cost $L$ and initial cost~$u_0$:
\begin{equation*}
\tag*{$\text{(CV)}_{t,x}$}
\min_{\xi\in AC([0,t],\mathbb{R}^N)}\bigg\lbrace{u_{0}(\xi(0))+\int_{0}^{t}L(\xi(s),\dot{\xi}(s))ds\ \Big|\ \xi(t)=x\ \bigg\rbrace}
\end{equation*}
provides 
the (unique) viscosity solution of the Hamilton-Jacobi equation~\eqref{HJ} with initial data
$u(0,\cdot)=u_{0}$
(see~\cite[Section~6.4]{CS}).
We recall below some properties of viscosity solutions 
of interest in this paper which follow from the
representation formula~\eqref{Value-function} 
(cfr.~\cite[Sections 1.2, 6.3, 6.4]{CS}). 

\begin{theorem}
\label{prop-viscsol}
Assume that the Legendre transform $L$ of $H$ in~\eqref{Legendre-L} satisfies the
assumptions \textbf{(L1)-(L2)} and, given $u_{0}\in \Lip(\mathbb{R}^N)$,
 let $u$ be the viscosity solution of~\eqref{HJ}-\eqref{in-data}
on $[0,+\infty[\,\times\R^N$,
defined by~\eqref{Value-function}.
Then, the following holds true.

\begin{enumerate}
\item[$(i)$] \textbf{Dynamic programming principle}: for all $x\in\R^N$ and $s\in [0,t[\,$, $t>0$, we have 
\begin{equation}
\label{DPP}
u(t,x)=\min_{\xi\in AC([s,t],\mathbb{R}^N)}\bigg
\lbrace{u(s,\xi(s))+\int_{s}^{t}L(\xi(\tau),\dot{\xi}(\tau))d\tau\ \Big|\ \xi(t)=x
\bigg\rbrace}.
\end{equation} 
Moreover, if $\xi^*$ is a minimizer for $(CV)_{t,x}$,
the restriction of $\xi^*$ to $[s, t]$
is also a minimizer in~\eqref{DPP}.
%

%
\item[$(ii)$] \textbf{Euler-Lagrange equation}:
for all $x\in\R^N$ and $t>0$, if $\xi^*$ is a minimizer for $(CV)_{t,x}$,
$\xi^*$ is a Caratheodory solution of the equation
\begin{equation}
\label{EL}
\frac{d}{ds}D_q L\big(\xi(s), \dot{\xi}(s)\big)=
D_x L\big(\xi(s), \dot{\xi}(s)\big)
\end{equation}
on $[0,t]$, i.e. $\xi^*$ satisfies~\eqref{EL} for almost every $s\in [0,t]$.
Moreover, one has 
\begin{equation}
\label{transvers}
D_q L \big(\xi^*(0), \dot{\xi^*}(0)\big)\in \partial u_0(\xi^*(0)),
\end{equation}
where $\partial u$ denotes Clarke's generalized gradient.
\item[$(iii)$] \textbf{Generalized backward characteristics}: 
for all $x\in\R^N$ and $t>0$, if $\xi^*$ is a minimizer for $(CV)_{t,x}$,
there exists $p^*\in AC([0,t],\mathbb{R}^N)$ (called the dual or co-state arc associated with $\xi^*$) so that $(\xi^*, p^*)$
provides the solution of the system
\begin{equation}
\label{hamilt-syst}
\begin{cases}
\hspace{.cm}
\dot \xi = D_p H(\xi, p),&
\\
\hspace{.cm}
\dot p = - D_x H(\xi, p),&
\end{cases}
\end{equation}
on $[0,t]$, with terminal condition
\begin{equation}
\label{final-data}
\begin{cases}
\hspace{.cm}
\xi (t) =x,&
\\
\hspace{.cm}
p (t)\in D^+_x u(t, x).&
\end{cases}
\end{equation}
Moreover, $u(s,\cdot)$ is differentiable at $\xi^*(s)$
for any $s\in\,]0,t[$ and one has
\begin{equation}
\label{p-cond1}
p^*(s)\in D_x u(s,\xi^*(s))\qquad\forall~s\in\,]0,t[\,,
\end{equation}
\begin{equation}
\label{p-cond0}
p^*(0)\in \partial u_0(\xi^*(0))\,.
\end{equation}
%
%
\end{enumerate} 
\end{theorem}
\quad
\\
\noindent
By the above observations,
the family of nonlinear operators
$$S_t:\Lip_{loc}(\mathbb{R}^N)\rightarrow\Lip_{loc}(\mathbb{R}^N), 
\qquad\quad u_0\mapsto S_t u_0, \ \ t\geqslant 0,$$
defined by
\begin{equation}
\label{value-function}
S_t u_0(x)
\doteq
\min_{\xi\in AC([0,t],\mathbb{R}^N)}\bigg\lbrace{u_{0}(\xi(0))+\int_{0}^{t}L(\xi(s),\dot{\xi}(s))ds\ \Big|\ \xi(t)=x\ \bigg\rbrace}
\end{equation}
enjoy the following properties:
\begin{itemize}
\item[(i)] for every  $u_0\in\Lip(\mathbb{R}^N)$, $u(t,x)\doteq S_t u_0(x)$ provides the unique viscosity solution of the 
Cauchy problem~\eqref{HJ}-\eqref{in-data};
\item[(ii)] (semigroup property)
$$S_{t+s} u_0 = S_t\, S_s u_0
\,,\quad\forall~t,s\geqslant 0\,,\;\forall u_0\in\Lip(\mathbb{R}^N).$$
%
\end{itemize}
\medskip

\section{Upper compactness estimates}
\label{sec:up-ext}

\subsection{A-priori bounds on the value function}
\label{subsec:bounds-vf}
Let $H:\R^N\to \R$ be  a function satisfying  the assumptions {\bf (H1)-(H2)}
and let $L$ be the corresponding Legendre transform in~\eqref{Legendre-L} .
We  establish here an a-priori bound  on the support of $S_t u_0 - S_t 0$
in terms of the support of $u_0$, and we
collect  some a-priori local bounds  on the 
semiconcavity costant and on the gradient of the value function $S_t u_0$
in~\eqref{value-function}.
In particular, given $R,M>0$, we shall  derive such properties in connection
with the set of initial data introduced in~\eqref{CLM-1}:
\begin{equation*}
\mathcal{L}_{[R,M]}=
\Big\lbrace{u_0\in \Lip(\mathbb{R}^N)\ \big|\ \mathrm{supp}(u_0)\subset [-R,R]^N\,,\; 
\Lip[u_0]\leqslant M\Big\rbrace}.
\end{equation*}

\begin{lemma}\label{L1-bound-xi}
Assume that the Legendre transform $L$ of $H$ 
satisfies the
assumptions \textbf{(L1)-(L2)} and, given $u_{0}\in \Lip(\mathbb{R}^N)\cap\mathbf{L}^\infty(\mathbb{R}^N)$,
 let $u$ be the viscosity solution of~\eqref{HJ}-\eqref{in-data}
on $[0,+\infty[\,\times\R^N$,
defined by~\eqref{value-function}.
Then, given $(t,x)\in \,]0,\infty[\,\times\mathbb{R}^{N}$, and letting
$\xi^*$ be a minimizer for $(CV)_{t,x}$, 
one has
\begin{equation}
\label{est-char-1}
\big|x-\xi^*(\tau)\big|
\leq l_{1}\big(\|u_0\|_{\strut\mathbf{L}^\infty(\mathbb{R}^N)},t\big)+\frac{|x|}{2}
\qquad\ \forall~\tau\in [0,t], 
\end{equation}
where
\begin{equation}
\label{l1-def}
 l_{1}\big(\|u_0\|_{\strut\mathbf{L}^\infty(\mathbb{R}^N)},t\big)\doteq
\frac{\|u_0\|_{\strut\mathbf{L}^\infty(\mathbb{R}^N)} + (c_1+c_3) \, t}{1+c_1\, t}+
t\, \bigg(\frac{2\, (1+c_1\,t )}{c_2}\bigg)^{\!\!\frac{1}{\alpha-1}}\,.
\end{equation}
Moreover, if $x\in [-l, l]^N$, $l>0$, then there holds
\begin{equation}
\label{est-der-char-1}
\big|\dot{\xi^*}(\tau)\big|
\leq \Chi_1\big(\|u_0\|_{\strut\mathbf{L}^\infty(\mathbb{R}^N)}, \Lip[u_0], l,t\big)
\qquad\ \forall~\tau\in [0,t], 
\end{equation}
for some constant $\Chi_1(\|u_0\|_{\strut\mathbf{L}^\infty}, \Lip[u_0],l,t)>0$
depending on $\|u_0\|_{\strut\mathbf{L}^\infty}$, 
$\Lip[u_0],l,t$.
\end{lemma}
\begin{proof} \

\noindent
{\bf 1.} Given $(t,x)\in [0,\infty[\,\times\mathbb{R}^{N}$, let
$\xi^*$ be a minimizer for $(CV)_{t,x}$.
First observe that since 
\begin{equation}
\nonumber
z \geqslant \bigg(\frac{2(1+c_1 t)}{c_2}\bigg)^{\!\!\frac{1}{\alpha-1}}
\quad\Longrightarrow\quad 
\frac{c_2}{2(1+c_1 t)} \, z^\alpha \geqslant  z \,,
\end{equation}
we deduce 
\begin{equation}
\label{est-xi-1}
\begin{aligned}
\big|x-\xi^*(\tau)\big|&=\big|\xi^*(t)-\xi^*(\tau)\big|
\leq \int_0^t \big|\dot\xi^*(s)\big| ds 
\\
&\leq \frac{c_2}{2(1+c_1 t)}  \int_0^t \big|\dot\xi^*(s)\big|^\alpha ds + 
t \, \bigg(\frac{2\, (1+c_1 t)}{c_2}\bigg)^{\!\!\frac{1}{\alpha-1}}\, \, 
\qquad\ \forall~\tau\in [0,t], \, x\in\R^N.
\end{aligned}
\end{equation}
Towards an estimation of $\int_0^t \big|\dot\xi^*(s)\big|^\alpha ds$, relying on~\eqref{subl-L2}
we derive
\begin{equation}
\label{est-xi-2}
c_2 \int_0^t \big|\dot\xi^*(s)\big|^\alpha ds \leq \int_0^t L(\xi^*(s),\dot{\xi^*}(s))ds+ c_3\, t\,.
\end{equation}
On the other hand, by definition~\eqref{value-function} and thanks to~\eqref{subl-L1},
we have
%
\begin{align}
\nonumber
u(t,x)&\geq 
-\|u_0\|_{\strut\mathbf{L}^\infty(\mathbb{R}^N)} +  \int_0^t L(\xi^*(s),\dot{\xi^*}(s))ds
\\
\noalign{\bigskip}
\nonumber
u(t,x)&\leq 
\|u_0\|_{\strut\mathbf{L}^\infty(\mathbb{R}^N)} + t\cdot L(x,0)
\\
\noalign{\smallskip}
\nonumber
&\leq
\|u_0\|_{\strut\mathbf{L}^\infty(\mathbb{R}^N)} + t\cdot
c_1(1+|x|),
\end{align}
%
which, together, yield
\begin{equation}
\label{est1-value-funct}
 \int_0^t L(\xi^*(s),\dot{\xi^*}(s))ds \leq 2 \|u_0\|_{\strut\mathbf{L}^\infty(\mathbb{R}^N)} + t\cdot
c_1(1+|x|).
\end{equation}
Thus, combining~\eqref{est-xi-2}, \eqref{est1-value-funct}, we find
\begin{equation}
\label{est-xi-3}
c_2 \int_0^t \big|\dot\xi^*(s)\big|^\alpha ds \leq 
2 \|u_0\|_{\strut\mathbf{L}^\infty(\mathbb{R}^N)} + (c_1+c_3) \, t + c_1\cdot |x|\, t\,.
\end{equation}
Hence, from~\eqref{est-xi-1}, \eqref{est-xi-3}, we deduce
\begin{equation}
\label{est-xi-4}
\big|x-\xi^*(\tau)\big|\leq
 \frac{\|u_0\|_{\strut\mathbf{L}^\infty(\mathbb{R}^N)}+(c_1+ c_3) \, t}{1+c_1 t} 
+t\,\bigg(\frac{2\, (1+c_1 t)}{c_2}\bigg)^{\!\!\frac{1}{\alpha-1}}
+\frac{|x|}{2}
\qquad\ \forall~\tau\in [0,t], 
\end{equation}
which proves~\eqref{est-char-1}. 
\\
{\bf 2.}
Towards a bound on $|\dot\xi^*|$, 
observe that, if $x\in [-l, l]^N$, because of~\eqref{est-char-1} we find
\begin{equation}
\label{est-xi-9}
\begin{aligned}
|\xi^*(\tau)|&\leq \big|\xi^*(\tau)-x\big|+|x|
\\
&\leq 
 l_{1}\big(\|u_0\|_{\strut\mathbf{L}^\infty(\mathbb{R}^N)},t\big)+\frac{3\, |x|}{2}
\\
&\leq 
 l_{1}\big(\|u_0\|_{\strut\mathbf{L}^\infty(\mathbb{R}^N)},t\big)+\frac{3\sqrt{N}\cdot  l}{2}
\end{aligned}
\quad \qquad\forall~\tau\in[0,t]\,.
\end{equation}
On the other hand, thanks to the convexity of $L(x,q)$
with respect to~$q$, relying on~\eqref{subl-L1}-\eqref{subl-L3}, \eqref{EL}, \eqref{transvers},~\eqref{est-xi-3}, \eqref{est-xi-9},
and recalling Remark~\ref{clarke-grad}, we derive
\begin{equation}
\begin{aligned}
\label{est-xi-5}
c_2 \big|\dot{\xi^*}(\tau)\big|^{\strut\alpha} &\! \leq\! c_3 + L\big(\xi^*(\tau), \dot{\xi^*}(\tau)\big)
\\
&\! \leq\!  c_3 + L\big(\xi^*(\tau), 0\big) + \big\langle D_q L\big(\xi^*(\tau), \dot{\xi^*}(\tau)\big),\, \dot{\xi^*}(\tau)\big\rangle 
\\
&\! \leq\!  c_1+ c_3 + c_1 \big|\xi^*(\tau)\big|+ \bigg\langle 
\int_0^\tau \frac{d}{ds}D_q L\big(\xi^*(s), \dot{\xi^*}(s)\big)ds +
D_q L\big(\xi^*(0), \dot{\xi^*}(0)\big),\, \dot{\xi^*}(\tau)\bigg\rangle 
\\
&\! \leq\!  c_1\!\cdot\! \Big(1\!+ \! l_{1}\Big(\!\|u_0\|_{\strut\mathbf{L}^\infty},t\Big)\!\Big)\!+\!\frac{3\sqrt{N}c_1 \!\cdot\! l}{2}
+ c_3 \! + \!
\bigg(\!\int_0^\tau \!\!\big|D_x L\big(\xi^*(s), \dot{\xi^*}(s)\big)\big|ds \!+\!\Lip[u_0]\bigg)\!\cdot\! \big|\dot{\xi^*}(\tau)\big|
\\
&\! \leq\!  c_1\!\cdot\! \Big(1\!+ \! l_{1}\Big(\!\|u_0\|_{\strut\mathbf{L}^\infty},t\Big)\!\Big)\!+\!\frac{3\sqrt{N}c_1 \!\cdot\! l}{2}
+ c_3 \! + \!
\bigg(\!c_4 \int_0^t \!\!\big|\dot{\xi^*}(s)\big|^\alpha ds +c_5 \, t +\Lip[u_0]\bigg)\!\cdot\! \big|\dot{\xi^*}(\tau)\big|
\\
&\! \leq\!  c_1\!\cdot\! \Big(1\!+ \! l_{1}\Big(\!\|u_0\|_{\strut\mathbf{L}^\infty},t\Big)\!\Big)\!+\!\frac{3\sqrt{N}c_1 \cdot l}{2}
+ c_3  + 
\\
&
+\Bigg(\frac{2\,c_4 \|u_0\|_{\strut\mathbf{L}^\infty}}{c_2}+ \frac{(c_1+c_3)c_4}{c_2} \, t + \frac{c_1\,c_4\cdot l\, t}{c_2}
+c_5 \, t +\Lip[u_0]\Bigg)\!\cdot\! \big|\dot{\xi^*}(\tau)\big|
\qquad\ \forall~\tau\in [0,t].
\end{aligned}
\end{equation}
Then, setting
\begin{equation}
\begin{gathered}
b_1\doteq b_1(l,t)= \frac{c_1}{c_2}\Big(1+\frac{3\sqrt{N}\, l}{2}\Big)+\frac{c_1+c_3}{c_2}
+\frac{c_1}{c_2} \cdot l_{1}\Big(\!\|u_0\|_{\strut\mathbf{L}^\infty},t\Big)\,, 
\\
\noalign{\medskip}
b_2\doteq 
b_2\big(\|u_0\|_{\strut\mathbf{L}^\infty(\mathbb{R}^N)}, \Lip[u_0], l, t\big)
= 
\frac{(c_1+c_3+c_1\cdot l)c_4+c_2\, c_5}{c_2^2} \, t +
\frac{2\,c_4 \|u_0\|_{\strut\mathbf{L}^\infty}}{c_2^2} + \frac{\Lip[u_0]}{c_2}\,,
\end{gathered}
\end{equation}
we obtain from~\eqref{est-xi-5} the estimate~\eqref{est-der-char-1}
with
\begin{equation}
\label{chi1-def}
\Chi_1\big(\|u_0\|_{\strut\mathbf{L}^\infty(\mathbb{R}^N)}, \Lip[u_0], l,t\big)\doteq 
\max\big\{b_1,\ (1+b_2)^{\frac{1}{\alpha-1}}
\big\}.
\end{equation}
%
\end{proof}
\begin{remark}
\label{l1bound-St0}
\rm
By the proof of Lemma~\ref{L1-bound-xi}
one deduces that the following further properties hold
\begin{itemize}
\item[(i)] Given any $l>0$, there exist constants $\tau_1(l), c_6(l)>0$
such that, for every $u_{0}\in \Lip(\mathbb{R}^N)\cap\mathbf{L}^\infty(\mathbb{R}^N)$
with $\|u_0\|_{\strut\mathbf{L}^\infty(\mathbb{R}^N)}\leq c_6(l)$,
the following hold:
\begin{itemize}
\item[--]
letting $\xi^*$ be any minimizer for~$(CV)_{t,x}$, $t\leq \tau_1(l)$, $x\in[-l,l]^N$, 
one has 
\begin{equation}
\label{xistar-bounds3}
\big|\xi^*(s)\big|\leq 2 l
\qquad\quad\forall~s\in[0,t]\,;
\end{equation}
\item[--]
letting $\xi^*$ be any  maximizer for
\begin{equation*}
\tag*{$\text{(CV)}^{t,x}$}
\max_{\xi\in AC([t,\tau_1(l)],\mathbb{R}^N)}\bigg\lbrace{u_0(\xi(\tau_1(l)))-
\int_t^{\tau_1(l)}L(\xi(s),\dot{\xi}(s))ds\ \Big|\ \xi(t)=x\ \bigg\rbrace},
\end{equation*}
with $t\leq \tau_1(l)$, $\xi^*( \tau_1(l))\in[-l,l]^N$, 
one has 
\begin{equation}
\label{xistar-bounds3b}
\big|\xi^*(s)\big|\leq 2 l
\qquad\quad\forall~s\in[t,\tau_1(l)]\,.
\end{equation}
\end{itemize}
\item[(ii)] Given any $l>0$, there exist constants $\tau_2(l), c_7(l)>0$
such that, for every $u_{0}\in \Lip(\mathbb{R}^N)\cap\mathbf{L}^\infty(\mathbb{R}^N)$
with $\|u_0\|_{\strut\mathbf{L}^\infty(\mathbb{R}^N)}\leq c_7(l)$,
letting $\xi^*$ be any minimizer for~$(CV)_{t,x}$, $t\leq \tau_2(l)$, $x\in\R^N\setminus \,]-l,l[^N$, one has 
\begin{equation}
\label{xistar-bounds4}
\big|\xi^*(s)\big|> \frac{|x|}{3}
\qquad\quad\forall~s\in[0,t]\,.
\end{equation}
\end{itemize}
Moreover, by definition~\eqref{value-function},
relying on~\eqref{subl-L1}, \eqref{subl-L2}
and on~\eqref{est-der-char-1},
one can show that:
\begin{itemize}
\item[(iii)] Given any $l>0$, there exists a constant $c_8(l)>0$
depending on $l$ such that there holds
\begin{equation}
\label{linf-St0-bound}
\big|S_t 0(x)\big|\leq c_8(l)\cdot t\qquad\quad\forall~x\in[-l, l\,]^N,
\ t >0\,. 
\end{equation} 
\end{itemize}
\end{remark}
\medskip

\begin{corollary}
\label{Supp} 
Let $H:\R^N\to \R$ be  a function satisfying  the assumptions~{\bf (H1)-(H2)}
and $\{S_t : \Lip (\mathbb{R}^N) \to \Lip (\mathbb{R}^N)\}_{t\geqslant 0}$ be the semigroup of viscosity solutions generated 
by~\eqref{HJ}. 
Then, 
letting $ \mathcal{L}_{[R,M]}$, $R,M>0$, 
be the set in~\eqref{CLM-1},
for any $u_{0}\in \mathcal{L}_{[R,M]}$, $t>0$, one has 
\begin{equation}
\label{supp-est-2}
\mathrm{supp}(S_t u_0 - S_t 0)\subseteq [-l_2 (R, M,t),\,  l_2 (R, M,t)]^N,
\end{equation}
where 
\begin{equation}
\label{compss}
l_2 (R, M,t)\doteq 2 \Bigg(
\frac{M\sqrt{N}\cdot R+(c_1+c_3)\,t}{1+c_1\,t }+t\, \bigg(\frac{2\, (1+c_1\,t )}{c_2}\bigg)^{\!\!\frac{1}{\alpha-1}}
\Bigg)+ 2\sqrt{N}\cdot R.
\end{equation}
%
%
\end{corollary}
\begin{proof}
As observed in Section~\ref{subsec:HJ}, the Legendre transform $L$ of $H$ 
satisfies the
assumptions \textbf{(L1)-(L2)}. Thus, given $(t,x)\in [0,\infty[\,\times\mathbb{R}^{n}$,  
and $u_{0} \in \mathcal{L}_{[R,M]}$,
we can apply Lemma~\ref{L1-bound-xi} for the viscosity solutions to~\eqref{HJ}
with initial data $u_0$ and $0$.
Notice that, recalling~\eqref{CLM-1},
$u_{0}\in \mathcal{L}_{[R,M]}$  implies 
\begin{equation}
\label{supp-u0-20}
\mathrm{supp}(u_0)\subset [-R,R]^N
\end{equation}
and 
\begin{equation}
\label{ellinf-bound-1}
|u_0(x)|\leq M \sqrt{N}\cdot R
\qquad\ \forall~x\in\R^N.
\end{equation}
Then, employing the notation in~\eqref{l1-def}, set 
\begin{equation}
\label{l1bis-def}
l_1(u_0, t)\doteq l_{1}\big(M \sqrt{N}\cdot R,\,t\big)\,.
\end{equation}
Hence, relying on~\eqref{est-char-1} and recalling~\eqref{value-function} we find
\begin{equation}
\label{Stu0v0}
\begin{aligned}
S_t u_0 (x) &\doteq 
\min_{\xi\in AC([0,t],\mathbb{R}^N)}\bigg\{u_{0}(\xi(0))+\int_{0}^{t}L(\xi(s),\dot{\xi}(s))ds\ \Big|\ \xi(t)=x,
\\
&\hspace{2.8in} |x-\xi(0)|\leq l_1(u_0, t)+\frac{|x|}{2}\ \bigg\},
\\
\noalign{\bigskip}
S_t\, 0 (x) &\doteq 
\min_{\xi\in AC([0,t],\mathbb{R}^N)}\bigg\{\int_{0}^{t}L(\xi(s),\dot{\xi}(s))ds\ \Big|\ \xi(t)=x,
\\
&\hspace{2.8in} |x-\xi(0)|\leq l_1(u_0, t)+\frac{|x|}{2}\ \bigg\}.
\end{aligned}
\end{equation}
Observe now that, for every given $\xi\in AC([0,t],\mathbb{R}^N)$, if 
\begin{equation}
\label{xi-est-5}
\big|x-\xi(0)\big|\leq l_1(u_0, t)+\frac{|x|}{2},
\qquad\quad
|x|\geq 2 l_1(u_0, t)+ 2 \sqrt{N}\cdot R\,,
\end{equation}
it follows
\begin{equation}
\nonumber
\label{}
\begin{aligned}
|\xi(0)|& \geq |x|- \big|x-\xi(0)\big|
\\
\noalign{\smallskip}
&\geq \frac{|x|}{2}-l_1(u_0, t)
\\
\noalign{\smallskip}
&\geq \sqrt{N}\cdot R\,.
\end{aligned}
\end{equation}
Therefore, by~\eqref{supp-u0-20} we deduce
that for all $x\in\R^N$ satisfying~\eqref{xi-est-5}, one has  
\begin{equation}
\nonumber
u_0(\xi(0))= 0,
\end{equation}
which, in turn, because of~\eqref{Stu0v0} implies
\begin{equation}
\label{Stu0v02}
S_t u_0 (x) = S_t\,0 (x)\qquad\qquad\forall ~x\in\R^N\setminus B\big(0, \, 2 l_1(u_0, t)+ 2 \sqrt{N}\!\cdot \!R\big).
\end{equation}
This, proves~\eqref{supp-est-2} since, comparing~\eqref{l1-def}, \eqref{compss}, \eqref{l1bis-def}, we have
\begin{equation}
l_2 (R, M,t) = 2 l_1(u_0,  t)+ 2 \sqrt{N}\!\cdot \!R.
\end{equation}
\end{proof}
\pagebreak

\begin{proposition}
\label{Properties of u} 
In the same setting of Corollary~\ref{Supp}, for every $u_0\in \Lip(\R^N)\cap \mathbf{L}^\infty(\mathbb{R}^N)$ and for any $l,T>0$,
the following properties hold true:
\begin{enumerate}
\item[(i)]  $S_T u_0$ is Lipschitz continuous on $[-l, l]^N$ with a Lipschitz constant
$\mu(\|u_0\|_{\strut\mathbf{L}^\infty}, \Lip[u_0],l,T)$
depending on $\|u_0\|_{\strut\mathbf{L}^\infty}$, 
$\Lip[u_0],l,T$.
%
\item [(ii)]  $S_T(u_0)$ is semiconcave on $]\!-l,\,  l[^{\strut N}$  with  a semiconcavity constant 
$\kappa(\|u_0\|_{\strut\mathbf{L}^\infty}, \Lip[u_0],l,T)$
depending on $\|u_0\|_{\strut\mathbf{L}^\infty}$, 
$\Lip[u_0],l,T$.
\end{enumerate}
\end{proposition}
\begin{proof}\

\noindent
{\bf 1.}
Given any $x\in [-l,l]^N$, let $\xi^*\in AC([0,T],\mathbb{R}^N)$ be a minimizer for $(CV)_{T,x}$,
so that one has 
\begin{equation}
\label{xix-eq1}
\xi^*(T)=x,
\qquad\qquad
S_{T} u_0(x)=u_{0}(\xi^*(0))+\int_{0}^{T}L(\xi^*(s),\dot{\xi^*}(s))ds\,.
\end{equation}
As observed in the proof of Corollary~\ref{Supp}, the Legendre transform $L$ of $H$ 
satisfies the
assumptions \textbf{(L1)-(L2)}
and thus we can apply Lemma~\ref{L1-bound-xi}.
Then, relying on~\eqref{est-xi-9}, \eqref{est-der-char-1},  we find
\begin{equation}
\label{xix-eq2}
\begin{aligned}
\big|\xi^*(\tau)\big|
&\leq r_{\scriptscriptstyle{T}}\doteq 
 l_{1}\big(\|u_0\|_{\strut\mathbf{L}^\infty(\mathbb{R}^N)},T\big)+\frac{3\, \sqrt{N}\cdot l}{2},
\\
\noalign{\smallskip}
\big|\dot{\xi^*}(\tau)\big|
&\leq q_{\scriptscriptstyle{T}}\doteq 
 \Chi_1\big(\|u_0\|_{\strut\mathbf{L}^\infty(\mathbb{R}^N)}, \Lip[u_0], l,T\big),
\end{aligned}
\quad \qquad\forall~\tau\in[0,T]\,.
\end{equation}
Next, set
\begin{equation}
\label{beta1-def}
\begin{aligned}
\beta_{\scriptscriptstyle{1,T}}&\doteq 
\sup\bigg\{
\max\big\{\big|D_x L(x,q)\big|,\big|D_q L(x,q)\big|\big\}\ \Big| \ |x|\leq r_{\scriptscriptstyle{T}}+2\, \sqrt{N}\cdot l, \ |q|\leq q_{\scriptscriptstyle{T}}
\bigg\}\,,
\end{aligned}
\end{equation}
and, given any $y\in  [-l,l]^N$, consider the map $\xi \in AC([0,T],\mathbb{R}^N)$
defined by
\begin{equation}
\label{xiy-def}
\xi (\tau)=\xi^*(\tau)+y-x\qquad\forall~\tau\in [0,T].
\end{equation}
Notice that, by~\eqref{xix-eq1}, \eqref{xix-eq2}, one has 
\begin{equation}
\label{xix-eq3}
\begin{gathered}
\xi(T)=y,
\\
\noalign{\smallskip}
\big|\xi(\tau)\big|
\leq r_{\scriptscriptstyle{T}}+2\, \sqrt{N}\cdot l,\qquad\quad
\big|\dot\xi(\tau)\big|
\leq q_{\scriptscriptstyle{T}}
\qquad\forall~\tau\in [0,T].
\end{gathered}
\end{equation}
Thus, recalling~\eqref{value-function} and because of~\eqref{xix-eq1}, \eqref{xix-eq2}, \eqref{xix-eq3}, we obtain
\begin{equation}
\label{lipest-1}
\begin{aligned}
S_{T} u_{0}(y)-S_{T} u_{0}(x)&\leq u_{0}(\xi(0))-u_{0}(\xi^{*}(0)) + 
\int_{0}^{T}\Big(L(\xi(s),\dot{\xi}(s))-L(\xi^{*}(s),\dot{\xi^*}(s))\Big)ds
\\
&\leq \Lip[u_0]\cdot \big|\xi(0)-\xi^*(0)\big|+
\int_{0}^{T}\beta_{\scriptscriptstyle{1,T}}\Big(\big|\xi(s)-\xi^*(s)\big|+\big|\dot\xi(s)-\dot{\xi^*}(s)\big|\Big)ds
\\
\noalign{\medskip}
&\leq \big(\Lip[u_0]+2\beta_{\scriptscriptstyle{1,T}}\,T\big)\cdot |y-x|\,.
\end{aligned}
\end{equation}
Performing a similar computation interchanging the role of $x$ and $y$, considering
a minimizer $\xi^*$  for $(CV)_{T,y}$ and a map $\xi \in AC([0,T],\mathbb{R}^N)$
as in~\eqref{xiy-def},
we find
\begin{equation}
\label{lipest-2}
S_{T}u_{0}(x)-S_{T}u_{0}(y)\leq \big(\Lip[u_0]+2\beta_{\scriptscriptstyle{1,T}}\,T\big)\cdot |y-x|.
\end{equation}
Thus, \eqref{lipest-1}-\eqref{lipest-2} together yield the Lipschitz continuity of $S_T u_0$ on
$[-l, l]^N$ with a Lipschitz constant
\begin{equation}
\label{mu-def}
\mu(\|u_0\|_{\strut\mathbf{L}^\infty}, \Lip[u_0],l,T)
\doteq \Lip[u_0]+2\beta_{\scriptscriptstyle{1,T}}\,T\,,
\end{equation}
proving (i).
\\

\noindent
{\bf 2.}
Given any $x, x-h, x+h \in [-l,l]^N$, let $\xi^*\in AC([0,T],\mathbb{R}^N)$ be a minimizer for $(CV)_{T,x}$,
and consider the maps $\xi^+, \xi^- \in AC([0,T],\mathbb{R}^N)$
defined by
\begin{equation}
\label{xih-def}
\xi^\pm (\tau)=\xi^*(\tau)\pm \frac{\tau}{T} \cdot h\qquad\forall~\tau\in [0,T].
\end{equation}
Notice that, by~\eqref{xix-eq1}, \eqref{xix-eq2}, one has 
\begin{equation}
\label{xix-eq4}
\begin{gathered}
\xi^\pm(0)=\xi^*(0),
\qquad\quad
\xi^\pm(T)=x\pm h,
\\
\noalign{\smallskip}
\big|\xi^\pm(\tau)\big|
\leq r_{\scriptscriptstyle{T}}+2\, \sqrt{N}\cdot l,\qquad\quad
\big|\dot{\xi^\pm}(\tau)\big|
\leq q_{\scriptscriptstyle{T}}+2\, \sqrt{N}\cdot \frac{l}{T}
\qquad\forall~\tau\in [0,T].
\end{gathered}
\end{equation}
Then, setting
\begin{equation}
\label{beta2-def}
\begin{aligned}
\beta_{\scriptscriptstyle{2,T}}&\doteq 
\sup\bigg\{
\max\big\{\big|D^2_{xx} L(x,q)\big|,\big|D^2_{qx} L(x,q)\big|,\big|D^2_{qq} L(x,q)\big|\big\}\ \Big| \ |x|\leq r_{\scriptscriptstyle{T}}+2\, \sqrt{N}\cdot l, \\
&\hspace{3.77in} |q|\leq q_{\scriptscriptstyle{T}}+2\, \sqrt{N}\cdot \frac{l}{T}
\bigg\}\,,
\end{aligned}
\end{equation}
recalling~\eqref{value-function} and relying on~\eqref{xix-eq1}, \eqref{xix-eq2}, \eqref{xix-eq4},  we have
\begin{equation}
\begin{aligned}
S_T u_0&(x+h)+S_T u_0(x-h)-2 S_T u_0(x)\leq
\\
\noalign{\smallskip}
&\leq \int_0^TL(\xi^+(s),\dot{\xi^+}(s))ds+ \int_0^TL(\xi^-(s),\dot{\xi^-}(s))ds-2 \int_0^TL(\xi^*(s),\dot{\xi^*}(s))ds
\\
\noalign{\smallskip}
&=\int_0^T  \bigg(L\Big(\xi^*(s)+\frac{s}{T}h,\dot{\xi^*}(s)\Big)+ L\Big(\xi^*(s)-\frac{s}{T}h,\dot{\xi^*}(s)\Big)-
2 L(\xi^*(s),\dot{\xi^*}(s))\bigg)ds \,+
\\
&\quad +\!\!\int_0^T\!\!\!\Big( L(\xi^+(s),\dot{\xi^+}(s))-L(\xi^+(s),\dot{\xi^*}(s))\Big)ds+\!\!\int_0^T\!\!\!\Big(L(\xi^-(s),\dot{\xi^-}(s))-
L(\xi^-(s),\dot{\xi^*}(s))\Big)ds
\\
\noalign{\smallskip}
&\leq\beta_{\scriptscriptstyle{2,T}}\cdot\!\!\int_0^T\!\!\!\Big(\frac{s}{T}\Big)^{\!2}|h|^2ds\,+
\\
&\quad +\int_0^T 
\!\!\!\bigg(
\int_0^1\Big(
\Big\langle D_q L \Big(\xi^+(s),\dot{\xi^*}(s)+\lambda\frac{h}{T}\Big),\frac{h}{T}\Big\rangle
+\Big\langle D_q L\Big(\xi^-(s),\dot{\xi^*}(s)-\lambda\frac{h}{T}\Big)\Big), \frac{-h}{T}\Big\rangle
\Big) d\lambda\bigg)ds
\\
\noalign{\smallskip}
&\leq\frac{\beta_{\scriptscriptstyle{2,T}}}{3}\,T\cdot \!|h|^2+\frac{|h|}{T}\cdot\!\!\int_0^T\!\!\!\bigg(
\int_0^1\Big| D_q L \Big(\xi^+(s),\dot{\xi^*}(s)+\lambda\frac{h}{T}\Big)-
D_q L\Big(\xi^-(s),\dot{\xi^*}(s)-\lambda\frac{h}{T}\Big) \Big| d\lambda\bigg)ds
\\
\noalign{\smallskip}
&\leq\frac{\beta_{\scriptscriptstyle{2,T}}}{3}\,T\cdot |h|^2+\beta_{2,T}\frac{|h|}{T}\cdot\int_0^T\bigg(\frac{2s}{T}|h|+\frac{2}{T}|h|\bigg)ds
\\
\noalign{\smallskip}
&\leq \beta_{\scriptscriptstyle{2,T}}\bigg(\frac{T}{3}+1+\frac{1}{T}\bigg)\cdot|h|^2.
\end{aligned}
\end{equation}
Thus,  $S_T u_0$ is semiconcave on
$]\!-l, l\,[^N$ with a semiconcavity constant
\begin{equation}
\label{kappa-def}
\kappa(\|u_0\|_{\strut\mathbf{L}^\infty}, \Lip[u_0],l,T)
\doteq
\beta_{\scriptscriptstyle{2,T}}\bigg(1+\frac{3+T^2}{3T}\bigg)\,,
\end{equation}
proving (ii).
\end{proof} 
\begin{remark}
\label{lipschitzianity}
\rm
By the proof  of Lemma~\ref{L1-bound-xi} and Proposition~\ref{Properties of u} it follows that,
under the same hypotheses of Corollary~\ref{Supp},  given any $c,l,m>0$ there
exists a constant $\tau_3(c,l,m)>0$
depending on $c, l$ and $m$,
such that the following holds. 
For every  
$u_0\in \Lip(\R^N)\cap \mathbf{L}^\infty(\mathbb{R}^N)$ 
 with $\|u_0\|_{\strut\mathbf{L}^\infty}\leq c$ and $\Lip[u_0]\leq m$,
the map $S_t u_0$ is Lipschitz continuous on $[-l, l]^N$ with 
$\Lip\big[S_t u_0; [-l,l]^N\big] \leq 2\cdot  m$, for all $t\leq \tau_3(c,l,m)$.
\end{remark}
\medskip

\medskip

\subsection{Continuity of the semigroup map ${\bf S_t},$ ${\bf t>0}$}

\noindent
It was shown in~\cite{APN} that, for every fixed $t\geq 0$, the map 
$S_t:\Lip(\mathbb{R}^N)\rightarrow\Lip_{loc}(\mathbb{R}^N)$ is continuous 
when the space 
$\Lip(\mathbb{R}^N)$ is endowed with the $\mathbf{W}^{1,1}_{loc}$-topology
and $S_t$ is restricted to  sets of functions with uniform Lipschitz constant
on bounded domains. The proof of this property was obtained in~\cite{APN}
exploiting the Hopf-Lax representation formula of solutions valid for Hamiltonians depending
only on the gradient of the solution. We shall extend here this result to
the case of Hamiltonians possibly depending also on the space variable
providing a direct proof of this property that relies only on the a-priori bounds on the solutions established in 
Section~\ref{subsec:bounds-vf}.
%
%
\begin{proposition}
\label{St-W11-continuity}
Let $u, u^\nu\in \Lip(\mathbb{R}^N)$  $(\nu \in \N)$ be such that
\begin{gather}
\label{unu-bound}
\|u^\nu\|_{\strut\mathbf{L}^\infty(\R^N)}\leqslant C\,,
\qquad\forall~\nu\,,
\qquad\text{for some} \ \ C>0\,,
\\
\noalign{\smallskip}
\label{unu-conv}
u^\nu \ \ \underset{\nu\to\infty}\longrightarrow \ \  u 
\qquad\text{in}\qquad \mathbf{W}^{1,1}_{loc}(\R^N)\,.
\end{gather}
Moreover, assume that
for every bounded domain $\Omega\subset\R^N$, there exists some constant $M_{\scriptscriptstyle{\Omega}}>0$
such that
\begin{equation}
\label{unu-bound2}
\Lip\big[u^\nu;\, {\Omega}\big]
\leqslant M_{\scriptscriptstyle{\Omega}}\qquad\forall~\nu\,.
\end{equation}
Then, for every fixed $t\geq 0$, one has
\begin{equation}
\label{conv-semigr}
S_t u^\nu\ \ \underset{\nu\to\infty}\longrightarrow \ \  S_t u 
\qquad\text{in}\qquad \mathbf{W}^{1,1}_{loc}(\R^N)\,.
\end{equation}
\end{proposition}
\begin{proof}\ \\
{\bf 1.}
In order to establish the proposition it will be sufficient to show that, 
given any bounded domain $\Omega\subset\R^N$, for any fixed $t> 0$, 
there holds
\begin{equation}
\label{conv-semigr-2}
S_t u^\nu \ \ \underset{\nu\to\infty}\longrightarrow \ \  S_t u 
\qquad\text{in}\qquad \mathbf{W}^{1,1}(\Omega)\,.
\end{equation}
Consider the set
\begin{equation}
\label{o-prime-def}
\Omega_\tau\doteq\Big\{x\in\R^N\, \big| \, d(x,\Omega)\leq l_1\big(C, \tau\big)+2\cdot  \big(\sup_{y\in\Omega}|y|+\mathrm{diam}(\Omega)\big)\Big\}\,,
\qquad\tau\in[0, t]\,,
\end{equation}
where $ l_1(C, \tau)$ is defined as in~\eqref{l1-def} with $C$ in place of $\|u_0\|_{\mathbf{L}^\infty}$.
Observe that, because of~\eqref{unu-bound}, \eqref{unu-bound2}, applying Lemma~\ref{L1-bound-xi}
we deduce that, for any $x\in\Omega$, letting
$\xi^*$ be a minimizer for $(CV)_{t,x}$, 
one has 
\begin{equation}
\label{xistar-bounds}
\begin{gathered}
\xi^*(\tau)+(y-x)\in\Omega_{t-\tau}\qquad \forall~y\in\Omega,
~\tau\in[0,t]\,,
\\
\noalign{\medskip}
|\dot{\xi^*}(\tau)|\leq \Chi_t
\qquad\ \forall~\tau\in[0,t]\,,
\end{gathered}
\end{equation}
for some constant $\Chi_t$ depending on $C, M_{\scriptscriptstyle{\Omega_t}}, \Omega_t, t$.
Then, with the same arguments of the proof of Proposition~\ref{Properties of u} we deduce that
there exist constants $\mu_t, \kappa_t>0$ so that:
\begin{itemize}
\item[(i)] $S_\tau u^\nu$ is Lipschitz continuous  on $\Omega_{t-\tau}$ with \ $\Lip[S_\tau u^\nu; \Omega_{t-\tau}]\leq \mu_t$ \ for all $\nu\in \N$, $\tau\in[0,t]$\,;
\item[(ii)] $S_t u^\nu$ is semiconcave on $\Omega$ with semiconcavity constant $\kappa_t$ \ for all $\nu\in\N$\,.
\end{itemize}
Thanks to the uniform bound in (i),  in order to prove that
\begin{equation}
\label{conv-semigr-nabla-L1}
\nabla S_t u^\nu \ \ \underset{\nu\to\infty}\longrightarrow \ \  \nabla S_t u 
\qquad\text{in}\qquad \mathbf{L}^{1}(\Omega)\,,
\end{equation}
it will be sufficient to show that
\begin{equation}
\label{grad-Stunu-conv}
\nabla S_t u^\nu (x)
\ \ \underset{\nu\to\infty}\longrightarrow \ \ 
\nabla S_t u (x)
\qquad\text{for \ a.e.}\ \ x\in\Omega\,.
\end{equation}
On the other hand, 
notice that because of~\eqref{unu-bound2}, 
we have
\begin{equation}
\label{equicont}
\big|u^\nu(x)-u^\nu(y)\big|\leq M_{\scriptscriptstyle{\Omega_t}} \cdot \big|x-y\big|\qquad\quad\forall~x,y\in\overline{\Omega_t}\,,
\qquad\quad\forall~\nu\,.
\end{equation}
Hence, relying on~\eqref{unu-bound}, \eqref{equicont}, and applying the Ascoli-Arzel\`a compactness theorem
we find that 
\begin{equation}
\label{univ-donv-unu}
u^\nu \ \ \underset{\nu\to\infty}\longrightarrow \ \ u\qquad\text{uniformly on}\quad\Omega_t\,.
\end{equation}
Observe now that, by~\eqref{xistar-bounds} and because of (i), the value of a solution
to~\eqref{HJ} with initial datum~$u^\nu$
depends
at any point  $(t,x), \, x\in\Omega$ only by the values of the Hamiltonian $H$ on the bounded
domain 
\begin{equation}
\bigg(\bigcup_{\tau\in[0,t]} \Omega\tau \bigg)\times B(0, \mu_t).
\end{equation}
Since the restriction of $H(x,p)$ to such a domain is uniformly Lipschitz continuous in both
variables, we may invoke the well-known contraction property of the semigroup map $S_t$ with respect to the
uniform convergence on compacta, which holds for Hamiltonians $H$ enjoying this property (e.g. see~\cite[Theorem~5.2.12]{CS}).
Thus, we deduce from~\eqref{univ-donv-unu} that 
\begin{equation}
\label{univ-donv-Stunu}
S_t u^\nu \ \ \underset{\nu\to\infty}\longrightarrow \ \ S_t u\qquad\text{uniformly on}\quad\Omega\,,
\end{equation}
which, in turn, implies
\begin{equation}
\label{elleuno-Stunu}
S_t u^\nu \ \ \underset{\nu\to\infty}\longrightarrow \ \ S_t u\qquad\text{in}\qquad \mathbf{L}^{1}(\Omega)\,.
\end{equation}
\\
\noindent
{\bf 2.}
Towards a proof of~\eqref{grad-Stunu-conv},
let $\Omega'$ be a subset of $\Omega$, with $\text{meas}(\Omega')=\text{meas}(\Omega)$,
where all $S_t u^\nu$, $\nu\in\N$, and $S_t u$ are differentiable.
Then, invoking properties (iv), (vi) stated in Theorem~\ref{Pro-Semi}, and relying on property~(ii) above,
we infer that, at every $x\in\Omega'$, there holds
\begin{equation}
\label{semiconc-Stnu}
S_t u^\nu(x+h) - S_t u^\nu(x) - \big\langle\nabla S_t u^\nu(x), h\big\rangle \leq \frac{\kappa_t}{2}\cdot  |h|^2
\end{equation}
for all $h\in\R^N$ such that $[x,x+h]\subset \Omega$. 
Since $\nabla S_t u^\nu(x)$ are uniformly bounded
by property (i) above, let $p\in\R^N$ be any accumulation point
i.e.  such that
\begin{equation}
\label{gradStnu-conv}
\nabla S_t u^{\nu_k} (x)
\ \ \underset{k\to\infty}\longrightarrow \ \  p\,,
\end{equation}
for some subsequence $\{\nabla S_t u^{\nu_k}\}_k$.
Then, taking the limit in~\eqref{semiconc-Stnu}
of $S_t u^{\nu_k}, \, \nabla S_t u^{\nu_k}$, as $k\to\infty$, 
and using~\eqref{elleuno-Stunu}, \eqref{gradStnu-conv}, we obtain
\begin{equation}
\label{semiconc-Stu}
S_t u(x+h) - S_t u(x) - \langle p, h\rangle \leq \frac{\kappa_t}{2}\cdot  |h|^2
\end{equation}
for all $h\in\R^N$ such that $[x,x+h]\subset \Omega$. 
Recalling Definition~\ref{general-grad} and by Theorem~\ref{Pro-Semi}-(iv), 
this inequality
implies that $p\in D^+ S_t u(x) = \{\nabla S_t u(x)\}$.
Since $p$ is an arbitrary accumulation point of $\{\nabla S_t u^\nu(x)\}_\nu$,
it follows that 
\begin{equation}
\nabla S_t u^\nu (x)
\ \ \underset{\nu\to\infty}\longrightarrow \ \ 
\nabla S_t u (x)
\qquad\forall~x\in\Omega'\,,
\end{equation}
which proves~\eqref{grad-Stunu-conv} and hence~\eqref{conv-semigr-nabla-L1}.
In turn, \eqref{conv-semigr-nabla-L1} together with~\eqref{elleuno-Stunu}
yields~\eqref{conv-semigr-2}, completing the proof of the proposition.
\end{proof}

\subsection{Conclusion of the proof of Theorem~\ref{upper-lower-estimate}-${\bf (i)}$}

Given $R,M,T>0$, consider  the set of initial data $\mathcal{L}_{[R,M]}$ introduced in~\eqref{CLM-1}.
Then, invoking Corollary~\ref{Supp}-(ii) we know that,
for every $u_0\in\mathcal{L}_{[R,M]}$,  there holds
\begin{equation}
\label{supp-est-22}
S_T u_0 = S_T 0 \quad \text{on} \quad \R^N\,\setminus \,]-l_T,\,  l_T[^N,
\end{equation}
with
\begin{equation}
\label{lt-def}
l_T= 2 \Bigg(
\frac{M\sqrt{N}\cdot R+(c_1+c_3)\,T}{1+c_1\,T }+T\, \bigg(\frac{2\, (1+c_1\,T )}{c_2}\bigg)^{\!\!\frac{1}{\alpha-1}}
\Bigg)+ 2\sqrt{N}\cdot R\,.
\end{equation}
On the other hand, relying on~\eqref{ellinf-bound-1}
and applying  Proposition~\ref{Properties of u}-(i) we find that $S_T u_0$
is a Lipschitz continuous map with Lipschitz constant
\begin{equation}
\label{muT-def}
\mu_T\doteq 
M+2\beta_{\scriptscriptstyle{1,T}}\,T\,,
\end{equation}
where
\begin{equation}
\label{beta1T-def}
\beta_{\scriptscriptstyle{1,T}}\doteq 
\sup\bigg\{
\max\big\{\big|D_x L(x,q)\big|,\big|D_q L(x,q)\big|\big\}\ \Big| \ |x|\leq \widetilde r_{\scriptscriptstyle{T}}, \ |q|\leq 
\widetilde q_{\scriptscriptstyle{1,T}}
\bigg\}\,,
\end{equation}
and 
\begin{equation}
\label{rTt-def}
\begin{aligned}
\widetilde r_{\scriptscriptstyle{T}}
&=
 l_{1}\big(M\sqrt{N}\cdot R,\,T\big)+\bigg(\frac{3\, \sqrt{N}}{2}+ 2\sqrt{N}\bigg)\cdot l_T
\\
\noalign{\smallskip}
&=
\frac{M\sqrt{N}\cdot R+ (c_1+c_3) \, T}{1+c_1\, T}+
T\, \bigg(\frac{2\, (1+c_1\,T )}{c_2}\bigg)^{\!\!\frac{1}{\alpha-1}}
+\frac{7\, \sqrt{N}\cdot l_T}{2}
\\
\noalign{\smallskip}
&=(1+7\sqrt{N})\bigg(\frac{M\sqrt{N}\cdot R+(c_1+c_3 )\, T}{1+c_1\, T}+
T\, \bigg(\frac{2\, (1+c_1\,T )}{c_2}\bigg)^{\!\!\frac{1}{\alpha-1}}\bigg)
+7\, N\cdot R\,,
\end{aligned}
\end{equation}
while
\begin{equation}
\label{qTt-def}
\begin{aligned}
\widetilde q_{\scriptscriptstyle{1,T}}
&\doteq 
 \Chi_1\big(M\sqrt{N}\cdot R,\,, M,\, l_T,\,T\big)
 +2\, \sqrt{N}\cdot l_T
\\
\noalign{\smallskip}
&=\max\Big\{b_1, (1+b_2)^{\frac{1}{\alpha-1}}
\Big\} +4\sqrt{N}\bigg(\frac{M\sqrt{N}\cdot R+(c_1+c_3) T}{1+c_1\, T}+
T\, \bigg(\frac{2\, (1+c_1\,T )}{c_2}\bigg)^{\!\!\frac{1}{\alpha-1}}\bigg)
\!+\!4\, N\!\cdot\! R,
\end{aligned}
\end{equation}
with
\begin{equation}
\label{b12-def}
\begin{aligned}
b_1
&= \frac{c_1}{c_2}\Big(1+\frac{3\sqrt{N}\, l_T}{2}\Big)+\frac{c_1+c_3}{c_2}
+\frac{c_1}{c_2} \cdot l_{1}\Big(M\sqrt{N}\cdot R,\,T\Big)
\\
\noalign{\smallskip}
&=\frac{c_1}{c_2}\Bigg(1+\big(1+3\sqrt{N}\big)\bigg(\frac{M\sqrt{N}\cdot R+(c_1+c_3) T}{1+c_1\, T}+
T\, \bigg(\frac{2\, (1+c_1\,T )}{c_2}\bigg)^{\!\!\frac{1}{\alpha-1}}\bigg)
\!+\!3N\!\cdot\! R\Bigg)\!+\!\frac{c_1+c_3}{c_2},
\\
\\
b_2&= 
\frac{1}{c_2^2} \Big(
\big(c_1 c_4 +c_3 c_4+c_2\, c_5\big) T + c_1  c_4\cdot l_T \cdot T+
\big(c_2+2\,c_4\sqrt{N}\cdot R\big) M\Big)
\\
\noalign{\smallskip}
&=
\frac{1}{c_2^2}
\Bigg[
\big(c_1 c_4 +c_3 c_4+c_2\, c_5\big) T 
+ \big(c_2+2\,c_4\sqrt{N}\cdot R\big) M \Bigg]+
\\
\noalign{\smallskip}
&\quad
+\frac{2 c_1 c_4}{c_2^2}
\Bigg[
\Bigg(
\frac{M\sqrt{N}\cdot R+(c_1+c_3)\,T}{1+c_1\,T }+T\, \bigg(\frac{2\, (1+c_1\,T )}{c_2}\bigg)^{\!\!\frac{1}{\alpha-1}}
\Bigg) T+ \sqrt{N}\cdot R\cdot T
\Bigg].
\end{aligned}
\end{equation}
Moreover, by Proposition~\ref{Properties of u}-(ii), using again~\eqref{ellinf-bound-1} we 
deduce  that $S_T u_0$ is a semiconcave map with semiconcavity constant
\begin{equation}
\label{kappaT-def}
\kappa_{\scriptscriptstyle{T}}\doteq 
\beta_{\scriptscriptstyle{2,T}}\bigg(1+\frac{3+T^2}{3T}\bigg),
\end{equation}
where
\begin{equation}
\label{beta2T-def}
\begin{aligned}
\beta_{\scriptscriptstyle{2,T}}&\doteq 
\sup\bigg\{
\max\big\{\big|D^2_{xx} L(x,q)\big|,\big|D^2_{qx} L(x,q)\big|,\big|D^2_{qq} L(x,q)\big|\big\}\ \Big| \ |x|\leq \widetilde r_{\scriptscriptstyle{T}}, \ \ |q|\leq \widetilde {q}_{\scriptscriptstyle{2,T}}
\bigg\}\,,
\end{aligned}
\end{equation}
with $\widetilde r_{\scriptscriptstyle{T}}$ defined in~\eqref{rTt-def} and
\begin{equation}
\label{qT2t-def}
\begin{aligned}
\widetilde q_{\scriptscriptstyle{2,T}}
&\doteq 
 \Chi_1\big(M\sqrt{N}\cdot R,\,, M,\, l_T,\,T\big)
 +2\, \sqrt{N}\cdot \frac{l_T}{T}
\\
\noalign{\smallskip}
&=\max\Big\{b_1, (1+b_2)^{\frac{1}{\alpha-1}}
\Big\} +\frac{4\sqrt{N}}{T}\bigg(\frac{M\sqrt{N}\cdot R+(c_1+c_3)\,T}{1+c_1\, T}+
T\, \bigg(\frac{2\, (1+c_1\,T )}{c_2}\bigg)^{\!\!\frac{1}{\alpha-1}}\bigg)
\!+\!\frac{4\, N\!\cdot\! R}{T},
\end{aligned}
\end{equation}
$b_1, b_2$ being the constants in~\eqref{b12-def}.
Therefore, recalling definition~\eqref{LSRM} we find that 
\begin{equation}
\label{inclusion1}
S_T u_0\!\restriction_{[-l_T,l_T]^N}\in \mathcal{SC}^{^{S_T 0}}_{[l_T,\,\mu_T,\kappa_T]}
\qquad\forall u_0\in\mathcal{L}_{[R,M]}\,,
\end{equation}
where $S_T u_0\!\restriction_{[-l_T,l_T]^N}$ denotes the restriction of $S_T u_0$
to the $N$-dimensional cube $[-l_T,l_T]^N$.
On the other hand, by~\eqref{supp-est-22} one has
\begin{equation}
\label{supp-est-23}
\text{supp}(\varphi)\subset [-l_T, l_T]^N\qquad\quad \forall~\varphi\in 
S_T(\mathcal{L}_{[R,M]})-S_T \,0\,.
\end{equation}
Moreover, by Proposition~\ref{Properties of u}-(ii) also $S_T \,0$ is a semiconcave map with semiconcavity 
constant $\kappa_T$.
Hence, applying Theorem~\ref{estSC}-(i), we deduce that, for $\varepsilon$ sufficiently small, there holds
\begin{equation}
\label{upper-entr-est-final}
\begin{aligned}
\mathcal{H}_{\epsilon}\Big(S_T(\mathcal{L}_{[R,M]})-S_T \,0\ \big|\ \mathbf{W}^{1,1}(\mathbb{R}^N)\Big)
&=
\mathcal{H}_{\epsilon}\Big(S_T(\mathcal{L}_{[R,M]})-S_T \,0\ \big|\ \mathbf{W}^{1,1}([-l_T, l_T]^N)\Big)
\\
\noalign{\smallskip}
&=
\mathcal{H}_{\epsilon}\Big(S_T(\mathcal{L}_{[R,M]}) \big|\ \mathbf{W}^{1,1}([-l_T, l_T]^N)\Big)
\\
\noalign{\smallskip}
&\leqslant
\mathcal{H}_{\epsilon}\Big(\mathcal{SC}^{^{S_T 0}}_{[l_T,\,\mu_T,\kappa_T]}\big|\ \mathbf{W}^{1,1}([-l_T, l_T]^N)\Big)
\\
\noalign{\smallskip}
&\leqslant \gamma^+_{_{[l_T,\mu_T,\kappa_T,N]}}\cdot\frac{1}{\varepsilon^N}\,,
\end{aligned}
\end{equation}
where
\begin{equation}
\gamma^+_{_{[l_T,\mu_T,\kappa_T,N]}}=
 \omega_N^N \cdot \Big(4N\cdot\big(1+ \mu_T+(\kappa_{\scriptscriptstyle{T}}+1)l_T\big)\Big)^{\!4N^2},
\end{equation}
with the constants $l_T,\,\mu_T,\kappa_T$  defined in~\eqref{lt-def}, \eqref{muT-def}, \eqref{kappaT-def}. 
This completes the proof of the upper bound~\eqref{Upper-est-H}.
\qed

\section{Lower compactness estimates}
\label{sec:low-est}

\subsection{Controllability of a class of semiconcave functions}

The proof of Theorem~\ref{upper-lower-estimate}-(ii) is based on a local controllability result 
for the class of semiconcave functions introduced in~\eqref{LSRM}.
Towards this goal,  we will first 
show
that a solution of~\eqref{HJ} with a semiconvex initial
condition preserves  the semiconvexity for a  time interval
that depends on the semiconvexity constant of the initial condition.
We shall obtain this property exploiting the representation
of a solution to~\eqref{HJ} as as the value function of the Bolza problem in the calculus of variations 
with running cost $L$ and initial cost~$u_0$:
\begin{equation*}
\tag*{$\text{(CV)}_{t,x}$}
\min_{\xi\in AC([0,t],\mathbb{R}^N)}\bigg\lbrace{u_{0}(\xi(0))+\int_{0}^{t}L(\xi(s),\dot{\xi}(s))ds\ \Big|\ \xi(t)=x\ \bigg\rbrace}.
\end{equation*}
\begin{proposition}
\label{le:semiconvexity}
Let  $H:\R^N\to \R$ be a function satisfying  the assumptions~{\bf (H1)-(H2)}
and \linebreak  $\{S_t : \Lip_{loc} (\mathbb{R}^N) \to \Lip_{loc} (\mathbb{R}^N)\}_{t\geqslant 0}$ be the semigroup of viscosity solutions 
generated by~\eqref{HJ}. 
Then, given any $c, l>0$ there exists a constant $\tau_4(c, l)>0$ depending on $c, l$,
such that the following holds. 
Given any $T<\tau_4(c, l)$, there exists $K_T>0$ such that, 
for every  semiconvex map
$u_0\in \Lip_{loc}(\R^N)\cap \mathbf{L}^\infty(\mathbb{R}^N)$ 
 with semiconvexity constant $K\leq K_T$
 and with $\|u_0\|_{\strut\mathbf{L}^\infty}\leq c$,
the following hold true.
\begin{itemize}
\item[$(i)$] $x\mapsto S_t u_0(x)$ is semiconvex on $[-l, l]^N$ for all $t\in [0,T]$.
\item[$(ii)$] $(t,x)\mapsto S_t u_0(x)$ is a $C^1$ classical solution of~\eqref{HJ} on~$]0,T]\times [-l,l]^N$.
\end{itemize}
\end{proposition}
\begin{proof}
Recall that, as observed in Section~\ref{subsec:HJ}, the Legendre transform $L$ of $H$ 
satisfies the
assumptions \textbf{(L1)-(L2)}. 
\\
{\bf 1.} 
Let $u_0\in \Lip_{loc}(\R^N)\cap \mathbf{L}^\infty(\mathbb{R}^N)$ 
be such that $\|u_0\|_{\strut\mathbf{L}^\infty}\leq c$.
Fix $T>0$, 
and observe that
letting
$\xi^*$ be any minimizer for~$(CV)_{t,x}$, $t\leq T$, $x\in[-l,l]^N$, 
thanks to Lemma~\ref{L1-bound-xi}
one has 
\begin{equation}
\label{xistar-bounds2}
\big|\xi^*(s)\big|\leq r_{T},\qquad \quad
|\dot{\xi^*}(s)|\leq \Chi_{T}
\qquad\quad\forall~s\in[0,t]\,,
\end{equation}
for some constants $r_{T}=r_{T}(c,l)$,
$\Chi_{T}=\Chi_{T}\big(c,\,
\Lip[u_0; [-r_t,r_t]^N], \, l\big)$ 
depending on $c$, \linebreak
$\Lip\big[u_0; [-r_{T},r_{T}]^N\big], l$ and $T$.
Then, with the same arguments of the proof of Proposition~\ref{Properties of u} we deduce that
there exist constants $\mu_T, C_T>0$ 
depending on  $c$,
$\Lip\big[u_0; [-r_{T},r_{T}]^N\big]$, $l$ and $T$, so that:
\begin{itemize}
\item[(a)] $S_t u_0$ is Lipschitz continuous  on 
%
\begin{equation}
\label{XI-def}
\Xi_t\doteq \bigg\{\xi^*(t)\ \big|\ \xi^* \ \text{is a minimizer for} \ (CV)_{T,x},\ x\in[-l, l] \bigg\},
\end{equation}
for all $t\in [0, T]$,
with Lipschitz constant $\mu_T$\,;
\item[(b)] $S_t u_0$ is semiconcave on $\Xi_t$
for all $t\in\,]0, T]$,
with semiconcavity constant $C_T$\,.
\end{itemize}
%
Hence, in oder to prove the proposition, we only have to show that
there exists $\tau_4>0$ so that, for any $T<\tau_4$
there holds
\begin{equation}\label{Uni2}
S_t u_0(x+h)+S_t u_0(x-h)-2 S_t u_0(x)\geqslant k_T\cdot |h|^2
\qquad\forall~t\leq T\,,
\end{equation}
for all $x,x-h,x+h\in [-l,l]^N$, and for some constant $k_T>0$. 
In fact, it follows from~\eqref{Uni2} that property~$(i)$ is verified
for all $T<\tau_4$.
On the other hand, once we know that $S_t u_0$ is both semiconcave and semiconvex on~$[-l, l]^N$,
for $t\in\,]0, T]$,
invoking Remark~\ref{visos-sol-differentiab} we immediately deduce that also the property $(ii)$ holds.
\quad
\\
\\
{\bf 2.} 
Towards a proof of~\eqref{Uni2},
given any $x,x-h,x+h\in [-l,l]^N$,  let $\xi^{\pm}$ be a minimizer
for~$(CV)_{t,x\pm h}$, $t\leq T$, and consider the map  $\xi_x \in AC([0,t],\mathbb{R}^N)$
defined by
\begin{equation}
\label{xix-def}
\xi_x (s)=\frac{\xi^+(s)+\xi^-(s)}{2}\qquad\forall~s\in [0,t].
\end{equation}
Then, one has
\begin{equation}
\label{xipm-t}
\xi^\pm(t)=x\pm h,\qquad\quad \xi_x(t)=x\,,
\end{equation}
and
\begin{equation}
\label{S_tu0-est1}
\begin{aligned}
S_t u_0(x\pm h)&=u_0(\xi^\pm(0))+\int_0^t
L\big(\xi^\pm(s),\dot{\xi^\pm}(s)\big)ds\,,
\\
\noalign{\medskip}
S_t u_0(x)&\leq u_0(\xi_x(0))+\int_0^t
L\big(\xi_x(s),\dot{\xi_x}(s)\big)ds\,.
\end{aligned}
\end{equation}
Moreover, recalling Theorem~\ref{prop-viscsol}-(iii), 
and because of property (a) above, there exist dual arcs $p^\pm \in AC([0,t],\mathbb{R}^N)$
so that $(\xi^\pm, p^\pm)$ provides the solution of
\begin{equation}
\label{hamilt-syst-2}
\begin{cases}
\hspace{.cm}
\dot \xi = D_p H(\xi, p),&
\\
\hspace{.cm}
\dot p = - D_x H(\xi, p),&
\end{cases}
\end{equation}
on $[0,t]$, with terminal condition
\begin{equation}
\label{final-data-2}
\begin{cases}
\hspace{.cm}
\xi^\pm (t) =x\pm h,&
\\
\hspace{.cm}
p (t)\in D^+ S_t u_0(x\pm h),&
\end{cases}
\end{equation}
that satisfy
\begin{equation}
\label{p-cond-2}
\big|p^\pm(s)\big|
=\big|DS_s u_0(\xi^\pm(s))\big|\leq \mu_T\qquad\forall~s\in\,]0,t[
\end{equation}
and
\begin{equation}
\label{p-cond0b}
p^\pm(0)\in \partial u_0(\xi^\pm(0))\,.
\end{equation}
Then, relying on~\eqref{xistar-bounds2}, \eqref{p-cond-2}, and setting
\begin{equation}
\label{beta3-def}
\beta_{\scriptscriptstyle{3,T}}
\doteq
\sup\bigg\{
\max\big\{\big|D^2_{px} H(x,p)\big|,\big|D^2_{pp} H(x,p)\big|\big\}\ \Big| \ |x|\leq r_T, \
 |p|\leq \mu_T
\bigg\}\,,
\end{equation}
we find
\begin{equation}
\label{xipmdot-est1}
\begin{aligned}
\big|\dot{\xi^+}(s)-\dot{\xi^-}(s)\big|&=\Big|D_p H(\xi^+(s),p^+(s)) - D_p H(\xi^-(s),p^-(s))\Big|
\\
\noalign{\smallskip}
&\leq \beta_{\scriptscriptstyle{3,T}}\Big(\big|\xi^+(s)-\xi^-(s)\big|+\big|p^+(s)-p^-(s)\big|\Big)
\qquad\forall~s\in\,]0,t[\,.
\end{aligned}
\end{equation}
Next, notice that, since $u_0$ is semiconvex with constant $K$, we have
\begin{equation}
\label{u0-semiconvex}
u_0(x+h)+u_0(x-h)-2u_0(x)\geqslant -K\cdot |h|^2
\qquad\forall~x,h\in\R^N\,.
\end{equation}
On the other hand, thanks to the convexity of $L(x,q)$
with respect to~$q$,
it follows
\begin{equation}
\label{L-convex-est}
\int_0^t L(\xi_x(s),\dot{\xi^+}(s))ds+ \int_0^t L(\xi_x(s),\dot{\xi^-}(s))ds-2 \int_0^t L(\xi_x(s),\dot{\xi_x}(s))ds\geq 0\,.
\end{equation}
Hence, setting
\begin{equation}
\label{beta4-def}
\beta_{\scriptscriptstyle{4,T}}
\doteq
\sup\bigg\{
\max\big\{\big|D^2_{xx} L(x,q)\big|,\big|D^2_{xq} L(x,q)\big|\big\}\ \Big| \ |x|\leq r_T, \
 |q|\leq \Chi_T
\bigg\}\,,
\end{equation}
and relying on~\eqref{xistar-bounds2}, \eqref{xipmdot-est1}, \eqref{L-convex-est},  we derive
\begin{equation}
\label{est-L-9}
\begin{aligned}
& \int_0^t L(\xi^+(s),\dot{\xi^+}(s))ds+ \int_0^t L(\xi^-(s),\dot{\xi^-}(s))ds-2 \int_0^t L(\xi_x(s),\dot{\xi_x}(s))ds
\\
\noalign{\medskip}
&\quad=
\int_0^t L(\xi_x(s),\dot{\xi^+}(s))ds+ \int_0^t L(\xi_x(s),\dot{\xi^-}(s))ds-2 \int_0^t L(\xi_x(s),\dot{\xi_x}(s))ds+
\\
&\qquad+\int_0^t\Big(L(\xi^+(s),\dot{\xi^+}(s))-L(\xi_x(s),\dot{\xi^+}(s))\Big)ds+
\int_0^t\Big( L(\xi^-(s),\dot{\xi^-}(s))-L(\xi_x(s),\dot{\xi^-}(s))\Big)ds
\\
\noalign{\medskip}
&\quad\geq -\beta_{\scriptscriptstyle{4,T}}
\int_0^t \bigg(\frac{\big|\xi^+(s)-\xi^-(s)\big|^2}{2}+\big|\xi^+(s)-\xi^-(s)\big|\big|\dot{\xi^+}(s)-\dot{\xi^-}(s)\big|\bigg)ds
\\
\noalign{\medskip}
&\quad\geq -\frac{\beta_{\scriptscriptstyle{4,T}}(1+3\beta_{\scriptscriptstyle{3,T}})}{2}
\int_0^t \bigg(\big|\xi^+(s)-\xi^-(s)\big|^2+
\big|p^+(s)-p^-(s)\big|^2\bigg)ds.
\end{aligned}
\end{equation}
Therefore, \eqref{S_tu0-est1}, \eqref{u0-semiconvex}, \eqref{est-L-9}, together yield
\begin{equation}
\label{S_tu0-est2}
\begin{aligned}
&S_t u_0(x+h)+S_t u_0(x-h)-2 S_t u_0(x)\geq 
\\
\noalign{\medskip}
&\hspace{0.5in}\geq u_0(x+h)+u_0(x-h)-2u_0(x)+
\\
&\hspace{1in} + \int_0^t L(\xi^+(s),\dot{\xi^+}(s))ds+ \int_0^t L(\xi^-(s),\dot{\xi^-}(s))ds-2 \int_0^t L(\xi_x(s),\dot{\xi_x}(s))ds
\\
\noalign{\medskip}
&\hspace{0.5in}\geq -\frac{\beta_{\scriptscriptstyle{4,T}}(1+3\beta_{\scriptscriptstyle{3,T}})}{2}
\int_0^t \bigg(\big|\xi^+(s)-\xi^-(s)\big|^2+
\big|p^+(s)-p^-(s)\big|^2\bigg)ds
\qquad\quad\forall~t\leq T.
\end{aligned}
\end{equation}
\quad
\\
\\
{\bf 3.} In order to recover the estimate~\eqref{Uni2} from~\eqref{S_tu0-est2}
we need to provide an upper bound on \, $\int_0^t |\xi^+(s)-\xi^-(s)|^2ds$ \, and \,
$\int_0^t |p^+(s)-p^-(s)|^2ds$. To this end observe first that, by the
same computations at point 2., because of~\eqref{xipm-t},~\eqref{xipmdot-est1},
for $0\leq t'<t$ we find
\begin{equation}
\label{xipm-est1}
\begin{aligned}
\big|\xi^+(t')-\xi^-(t')\big|^2&=\big|\xi^+(t)-\xi^-(t)\big|^2-\frac{1}{2} \int_{t'}^t
\frac{d}{ds} \big|\xi^+(s)-\xi^-(s)\big|^ 2ds
\\
\noalign{\medskip}
&=4|h|^2-\int_{t'}^t\bigg\langle
D_p H(\xi^+(s),p^+(s)) - D_p H(\xi^-(s),p^-(s)),
\xi^+(s)-\xi^-(s) \bigg\rangle ds
\\
\noalign{\medskip}
&\leq4|h|^2+\int_{t'}^t
\Big|D_p H(\xi^+(s),p^+(s)) - D_p H(\xi^-(s),p^-(s))\Big|
\cdot \big|\xi^+(s)-\xi^-(s)\big|ds
\\
\noalign{\medskip}
&\leq4|h|^2+\int_{t'}^t
\beta_{\scriptscriptstyle{3,T}}\Big(\big|\xi^+(s)-\xi^-(s)\big|+\big|p^+(s)-p^-(s)\big|\Big)
\cdot \big|\xi^+(s)-\xi^-(s)\big|ds
\\
\noalign{\medskip}
&\leq4|h|^2+
\frac{\beta_{\scriptscriptstyle{3,T}}}{2}  \int_0^t\big|p^+(s)-p^-(s)\big|^2ds
+\frac{3\beta_{\scriptscriptstyle{3,T}}}{2} 
\int_{t'}^t \big|\xi^+(s)-\xi^-(s)\big|^2ds.
\end{aligned}
\end{equation}
By a Gronwall type inequality, \eqref{xipm-est1} implies
\begin{equation}
\label{xipm-est2}
\int_0^t\big|\xi^+(s)-\xi^-(s)\big|^2ds\leq
\frac{8}{3\beta_{\scriptscriptstyle{3,T}}}
\Big(e^{\frac{3\beta_{\scriptscriptstyle{3,T}}}{2}\,t}-1\Big)
|h|^2+\frac{1}{3}
\Big(e^{\frac{3\beta_{\scriptscriptstyle{3,T}}}{2}\,t}-1\Big)  \int_0^t\big|p^+(s)-p^-(s)\big|^2ds.
\end{equation}
Towards an estimate of the second term in~\eqref{xipm-est2}, 
observe that, by~\eqref{unif-semiconc-L}, 
there holds
\begin{equation}
\label{unif-convex-H}
\qquad\qquad
\big\langle q, \, D^2_{pp} H(x,p)\,q\big\rangle\geq 
\beta_{\scriptscriptstyle{5,T}} |q|^2
\qquad \forall~
x\in B(0, r_T), \
 p\in B(0, \mu_T), \ q\in\R^N,
\end{equation}
with
\begin{equation}
\label{beta5-def}
\beta_{\scriptscriptstyle{5,T}}
\doteq
\inf\bigg\{
\big\langle q, \, D^2_{pp} H(x,p)\,q\big\rangle \ \Big| \ |x|\leq r_T, \
 |p|\leq \mu_T, \ |q|=1
\bigg\}>0.
\end{equation}
Hence, recalling that
$(\xi^\pm, p^\pm)$ are solutions of~\eqref{hamilt-syst-2},
defining the averaged matrices
\begin{equation}
\begin{aligned}
\widetilde H_{px}(s)&\doteq \widetilde H_{xp}(s)\doteq
\int_0^1D^2 H_{px} \big(
\lambda\xi^+(s)+(1-\lambda)\xi^-(s), \ \lambda p^+(s) +(1-\lambda)p^-(s)
\big)d\lambda,
\\
\noalign{\smallskip}
\widetilde H_{pp}(s)&\doteq
\int_0^1D^2 H_{pp} \big(
\lambda\xi^+(s)+(1-\lambda)\xi^-(s), \ \lambda p^+(s) +(1-\lambda)p^-(s)
\big)d\lambda,
\\
\noalign{\smallskip}
\widetilde H_{xx}(s)&\doteq 
\int_0^1D^2 H_{xx} \big(
\lambda\xi^+(s)+(1-\lambda)\xi^-(s), \ \lambda p^+(s) +(1-\lambda)p^-(s)
\big)d\lambda,
\end{aligned}
\end{equation}
and relying on~\eqref{xistar-bounds2}, \eqref{p-cond-2}, \eqref{unif-convex-H},
we get
\begin{equation}
\label{ppm-est1}
\begin{aligned}
&\frac{d}{ds}
\Big\langle p^+(s)-p^-(s),\, \xi^+(s)-\xi^-(s)\Big\rangle 
\\
\noalign{\smallskip}
&\qquad =
- \Big\langle D_x H(\xi^+(s),p^+(s)) - D_x H(\xi^-(s),p^-(s)),\, \xi^+(s)-\xi^-(s)\Big\rangle+
\\
&\qquad \qquad +\Big\langle p^+(s)-p^-(s),\, D_p H(\xi^+(s),p^+(s)) - D_p H(\xi^-(s),p^-(s))\Big\rangle
\\
\noalign{\medskip}
&\qquad = -
\Big\langle \widetilde H_{xx}(s) \big(\xi^+(s)-\xi^-(s)\big)+ \widetilde H_{xp}(s)
\big(p^+(s)-p^-(s)\big),\, \xi^+(s)-\xi^-(s) \Big\rangle+
\\
&\qquad \qquad +\Big\langle p^+(s)-p^-(s),\, \widetilde H_{px}(s) \big(\xi^+(s)-\xi^-(s)\big)+ \widetilde H_{pp}(s)
\big(p^+(s)-p^-(s)\big) \Big\rangle
\\
\noalign{\medskip}
&\qquad = -
\Big\langle \widetilde H_{xx}(s) \big(\xi^+(s)-\xi^-(s)\big),\, \xi^+(s)-\xi^-(s) \Big\rangle+
\Big\langle p^+(s)-p^-(s),\, \widetilde H_{pp}(s)
\big(p^+(s)-p^-(s)\big) \Big\rangle
\\
\noalign{\medskip}
&\qquad\geq
-\beta_{\scriptscriptstyle{3,T}}\big|\xi^+(s)-\xi^-(s) \big|^2 + \beta_{\scriptscriptstyle{5,T}}
\big|p^+(s)-p^-(s) \big|^2\,.
\end{aligned}
\end{equation}
Observe now that $u_0$ is semiconvex with constant $K$
and $S_tu_0$  is semiconcave on the domain $\Xi_t$
in~\eqref{XI-def}
with semiconcavity constant $C_T$.
Hence, invoking Proposition~\ref{prop2},
recalling Remark~\ref{clarke-grad}, 
and relying on~\eqref{sup-sub-diff-equal}, 
\eqref{xipm-t}, \eqref{final-data-2}, \eqref{p-cond0b}
we have
\begin{equation}
\label{ppm-est2}
\begin{aligned}
\Big\langle p^+(t)-p^-(t)&,\, \xi^+(t)-\xi^-(t)\Big\rangle
-\Big\langle p^+(0)-p^-(0),\, \xi^+(0)-\xi^-(0)\Big\rangle
\\
\noalign{\medskip}
&\leq C_T \big|\xi^+(t)-\xi^-(t)\big|^2 +K \big|\xi^+(0)-\xi^-(0)\big|^2
\\
\noalign{\medskip}
&\leq 4C_T |h|^2 + K \big|\xi^+(0)-\xi^-(0)\big|^2\,.
\end{aligned}
\end{equation}
On the other hand, by the same computations in~\eqref{xipm-est1}
we find
\begin{equation}
\label{xipm-est3}
\begin{aligned}
\big|\xi^+(0)-\xi^-(0)\big|^2&\leq4|h|^2+\int_0^t
\beta_{\scriptscriptstyle{3,T}}\Big(\big|\xi^+(s)-\xi^-(s)\big|+\big|p^+(s)-p^-(s)\big|\Big)
\cdot \big|\xi^+(s)-\xi^-(s)\big|ds
\\
\noalign{\medskip}
&\leq4|h|^2+
\frac{\epsilon\cdot \beta_{\scriptscriptstyle{3,T}}}{2}  \int_0^t\big|p^+(s)-p^-(s)\big|^2ds
+\bigg(1+\frac{1}{2\epsilon}\bigg)\beta_{\scriptscriptstyle{3,T}} 
\int_0^t \big|\xi^+(s)-\xi^-(s)\big|^2ds.
\end{aligned}
\end{equation}
Thus, \eqref{ppm-est2}, \eqref{xipm-est3}, together yield
\begin{equation}
\label{ppm-est3}
\begin{aligned}
\Big\langle p^+(t)-p^-(t)&,\, \xi^+(t)-\xi^-(t)\Big\rangle
-\Big\langle p^+(0)-p^-(0),\, \xi^+(0)-\xi^-(0)\Big\rangle
\\
\noalign{\medskip}
&\leq 4\big(C_T+K\big) |h|^2 + 
\frac{\beta_{\scriptscriptstyle{5,T}}}{2}  \int_0^t\big|p^+(s)-p^-(s)\big|^2ds+
\\
\noalign{\medskip}
&\quad + \bigg(1+\frac{K\,\beta_{\scriptscriptstyle{3,T}}}{2\beta_{\scriptscriptstyle{5,T}}}\bigg)K\, \beta_{\scriptscriptstyle{3,T}} 
\int_0^t \big|\xi^+(s)-\xi^-(s)\big|^2ds.
\end{aligned}
\end{equation}
Then, relying on~\eqref{ppm-est1}, \eqref{ppm-est3}, we derive
\begin{equation}
\label{ppm-est4}
\begin{aligned}
\beta_{\scriptscriptstyle{5,T}}
\int_0^t \big|p^+(s)-p^-(s) \big|^2 ds 
&\leq 
\beta_{\scriptscriptstyle{3,T}} \int_0^t\big|\xi^+(s)-\xi^-(s) \big|^2 ds
\\
&\quad +\int_0^t
\frac{d}{ds}
\Big\langle p^+(s)-p^-(s),\, \xi^+(s)-\xi^-(s)\Big\rangle ds
\\
\noalign{\medskip}
&\leq 4\big(C_T+K\big) |h|^2 + 
\frac{\beta_{\scriptscriptstyle{5,T}}}{2}  \int_0^t\big|p^+(s)-p^-(s)\big|^2ds+
\\
\noalign{\medskip}
&\quad +\bigg( \bigg(1+\frac{K\,\beta_{\scriptscriptstyle{3,T}}}{2\beta_{\scriptscriptstyle{5,T}}}\bigg)K+1\bigg)\, \beta_{\scriptscriptstyle{3,T}} 
\int_0^t \big|\xi^+(s)-\xi^-(s)\big|^2ds,
\end{aligned}
\end{equation}
which implies
\begin{equation}
\label{ppm-est5}
\int_0^t \big|p^+(s)-p^-(s) \big|^2 ds \leq 
\frac{8}{\beta_{\scriptscriptstyle{5,T}}} \big(C_T+K\big) |h|^2 + 
\Bigg( \bigg(1+\frac{K\,\beta_{\scriptscriptstyle{3,T}}}{2\beta_{\scriptscriptstyle{5,T}}}\bigg)K+1\Bigg)\frac{8\beta_{\scriptscriptstyle{3,T}} }{\beta_{\scriptscriptstyle{5,T}}}
\int_0^t \big|\xi^+(s)-\xi^-(s)\big|^2ds\,.
\end{equation}
\quad
\\
\\
{\bf 4.} 
Combining together~\eqref{xipm-est2} and \eqref{ppm-est5}, we find
\begin{equation}
\label{xipm-est6}
\begin{aligned}
&\int_0^t \bigg(\big|\xi^+(s)-\xi^-(s)\big|^2+
\big|p^+(s)-p^-(s)\big|^2\bigg)ds\leq
\\
\noalign{\medskip}
&\qquad \bigg(\frac{8}{3\beta_{\scriptscriptstyle{3,T}}}+
\frac{8}{3 \beta_{\scriptscriptstyle{5,T}}}\big(C_T+K\big)\bigg)
\Big(e^{\frac{3\beta_{\scriptscriptstyle{3,T}}}{2}\,t}-1\Big)|h|^2+
\\
\noalign{\medskip}
&\qquad+\frac{8\, \beta_{\scriptscriptstyle{3,T}}}{3\, \beta_{\scriptscriptstyle{5,T}}}
\Big(e^{\frac{3\beta_{\scriptscriptstyle{3,T}}}{2}\,t}-1\Big)
\Bigg( \bigg(1+\frac{K\,\beta_{\scriptscriptstyle{3,T}}}{2\beta_{\scriptscriptstyle{5,T}}}\bigg)K+1\Bigg)
\int_0^t \bigg(\big|\xi^+(s)-\xi^-(s)\big|^2+
\big|p^+(s)-p^-(s)\big|^2\bigg)ds.
\end{aligned}
\end{equation}
Observe now that, setting
\begin{equation}
\label{tau1-def}
\tau_4\doteq
\sup\Bigg\{
\inf\bigg\{\tau,\,
\frac{2}{3\,\beta_{\scriptscriptstyle{3,\tau}}}\ln\bigg(1+
\frac{3 \,\beta_{\scriptscriptstyle{5,\tau}}}{8\,\beta_{\scriptscriptstyle{3,\tau}}}\bigg)
\bigg\}\ \Big| \ \tau>0
\Bigg\},
\end{equation}
for any $T<\tau_4$ we can find $\mathcal{T}_{T}\geq T$  and $K_T>0$ such that,
for all $K\leq K_T$, one has
\begin{equation}
\label{t1-cond1}
T\leq 
\frac{2}{3\,\beta_{\scriptscriptstyle{3,\mathcal{T}_{_{T}}}}}\ln\bigg(1+
\frac{3 \,\beta_{\scriptscriptstyle{5,\mathcal{T}_{_{T}}}}^{\strut 2}}{4\,\beta_{\scriptscriptstyle{3,\mathcal{T}_{_{T}}}}\big(2\,\beta_{\scriptscriptstyle{5,\mathcal{T}_{_{T}}}}+2\,\beta_{\scriptscriptstyle{5,\mathcal{T}_{_{T}}}} K+ \beta_{\scriptscriptstyle{3,\mathcal{T}_{_{T}}}} K^2\big)}\bigg).
\end{equation}
Since
\begin{equation}
\label{t1-def1}
t\leq \frac{2}{3\,\beta_{\scriptscriptstyle{3,\mathcal{T}_{_{T}}}}}\ln\bigg(1+
\frac{3 \,\beta_{\scriptscriptstyle{5,\mathcal{T}_{_{T}}}}^{\strut 2}}{4\,\beta_{\scriptscriptstyle{3,\mathcal{T}_{_{T}}}}\big(2\,\beta_{\scriptscriptstyle{5,\mathcal{T}_{_{T}}}}+2\,\beta_{\scriptscriptstyle{5,\mathcal{T}_{_{T}}}} K+ \beta_{\scriptscriptstyle{3,\mathcal{T}_{_{T}}}} K^2\big)}\bigg)
\end{equation}
implies
\begin{equation}
\label{xipm-est7}
\frac{8\, \beta_{\scriptscriptstyle{3,\mathcal{T}_{_{T}}}}}{3\, \beta_{\scriptscriptstyle{5,\mathcal{T}_{_{T}}}}}
\Big(e^{\frac{3\beta_{\scriptscriptstyle{3,\mathcal{T}_{_{T}}}}}{2}\,t}-1\Big)
\Bigg( \bigg(1+\frac{K\,\beta_{\scriptscriptstyle{3,\mathcal{T}_{_{T}}}}}{2\beta_{\scriptscriptstyle{5,\mathcal{T}_{_{T}}}}}\bigg)K+1\Bigg)
\leq\frac{1}{2},
\end{equation}
we infer from~\eqref{xipm-est6}  that,  for any $t\leq T$, with $T$ satisfying~\eqref{t1-cond1}, 
and for $K\leq K_T$, there holds
\begin{equation}
\label{xipm-est8}
\int_0^t \bigg(\big|\xi^+(s)-\xi^-(s)\big|^2+
\big|p^+(s)-p^-(s)\big|^2\bigg)ds\leq
\frac{16\big(C_T+K\big)}{3}
\bigg(\frac{1}{\beta_{\scriptscriptstyle{3,\mathcal{T}_{_{T}}}}}+
\frac{1}{\beta_{\scriptscriptstyle{5,\mathcal{T}_{_{T}}}}}\bigg)
\Big(e^{\frac{3\beta_{\scriptscriptstyle{3,\mathcal{T}_{_{T}}}}}{2}\,T}-1\Big)|h|^2.
\end{equation}
Because of~\eqref{S_tu0-est2}, 
we deduce from~\eqref{xipm-est8} that,  for any $T<\tau_4$, with $\tau_4$ as in~\eqref{tau1-def},
and for $K\leq K_T$,
there holds
\begin{equation}
\label{S_tu0-est3}
\begin{aligned}
&S_t u_0(x+h)+S_t u_0(x-h)-2 S_t u_0(x)\geq 
\\
\noalign{\medskip}
&\hspace{0.5in}\geq -\frac{8\,\beta_{\scriptscriptstyle{4,\mathcal{T}_{_{T}}}}
(1+3\beta_{\scriptscriptstyle{3,\mathcal{T}_{_{T}}}})(C_T+K)}{3}
\bigg(\frac{1}{\beta_{\scriptscriptstyle{3,\mathcal{T}_{_{T}}}}}+
\frac{1}{\beta_{\scriptscriptstyle{5,\mathcal{T}_{_{T}}}}}\bigg)
\Big(e^{\frac{3\beta_{\scriptscriptstyle{3,\mathcal{T}_{_{T}}}}}{2}\,T}-1\Big)|h|^2
\qquad\quad\forall~t\leq T,
\end{aligned}
\end{equation}
which proves~\eqref{Uni2} and thus concludes the proof of the proposition.
\end{proof}
\begin{remark}
\label{semiconvexity}
\rm
By the proof  of Proposition~\ref{le:semiconvexity} it follows that,
under the same hypotheses of the proposition,  given any $c,l,K>0$ there
exists a constant $\tau_5(c,l,K)>0$
depending on $c, l$ and $K$,
such that the following holds. 
Given any $T\leq \tau_5(c,l,K)$, 
for every  semiconvex map
$u_0\in \Lip_{loc}(\R^N)\cap \mathbf{L}^\infty(\mathbb{R}^N)$ 
 with semiconvexity constant $K'\leq K$
 and with $\|u_0\|_{\strut\mathbf{L}^\infty}\leq c$,
the statements (i)-(ii) of Proposition~\ref{le:semiconvexity} hold.
\end{remark}
\medskip

\begin{proposition}\label{control-part}
In the same setting of Proposition~\ref{le:semiconvexity},
given any $R>0$, 
there exists $r_1(R)>0$ depending on $R$,
and, for every $r\leq r_1(R)$, $M>0$,
there exist $m_1(r,M)>0$, $k_1(r,m)>0$, $\tau_6(r,m)>0$ depending on $r \leq r_1(R)$, $m\leq m_1(r,M)$
and $M$, so that the following holds.
Letting $\mathcal L_{[R,M]}$, $\mathcal{SC}^{^{S_T 0,1}}_{[r,\,m,\,K]}$ be the sets defined 
in \eqref{CLM-1}, \eqref{LSRM-1}, for any $\psi \in \mathcal{SC}^{^{S_\tau 0, 1}}_{[r,\, 2m,\,2 K]}$,
with $r\leq r_1(R)$, $m\leq m_1(r,M)$, $\tau\leq \tau_6(r,m)$,
$K=k_1(r,m)$, 
there exists $u_0\in\mathcal L_{[R,M]}$ such that 
\begin{equation}
\label{exact-controll}
S_\tau u_0(x)=
\begin{cases}
\psi (x)\quad&\text{if}\qquad x\in [- r,  r]^N,
\\
\noalign{\smallskip}
S_\tau 0 (x) \quad&\text{if}\qquad x\in \R^N\setminus [- r, r]^N.
\end{cases}
\end{equation}
\end{proposition}
\begin{proof}
\quad
\\
{\bf 1.} 
Given $R,M>0$, fix $ r,  m>0$ such that
\begin{equation}
\label{rbar-mbar-conds1-2}
\begin{gathered}
r\leq r_1(R),\qquad\quad
m\leq 
 \min
 \bigg\{\frac{M}{10},\
\frac{1}{120\cdot \sqrt{N}\cdot  r}\cdot \min\Big\{c_6(36\cdot  r), \, c_7(6\cdot  r)\Big\}
\bigg\},
\end{gathered}
\end{equation}
where $c_6(36\cdot  r)$,  $c_7(6\cdot  r)>0$
are constants enjoying the properties stated in Remark~\ref{l1bound-St0}.
Moreover, 
 choose positive constants $\tau_1(36\cdot  r)$, 
$\tau_2(6\cdot  r)$, 
$\tau_3\big(c_6(36\cdot r), 72\cdot  r, m\big)$, 
$c_8(72\cdot  r)$,  
according with  Remark~\ref{l1bound-St0} and Remark~\ref{lipschitzianity},
so that:
\begin{itemize}
\item[a)] 
For every $\tau\leq \tau_1(36\cdot  r)$, $u_{0}\in \Lip(\mathbb{R}^N)\cap\mathbf{L}^\infty(\mathbb{R}^N)$
with $\|u_0\|_{\strut\mathbf{L}^\infty(\mathbb{R}^N)}\leq c_6(36\cdot  r)$,
the following holds:
\begin{itemize}
\item[--]
letting $\xi^*$ be any minimizer for~$(CV)_{t,x}$, $t\leq \tau$, 
one has 
\begin{equation}
\label{xistar-bounds3c}
\begin{aligned}
x\in[- r,  r]^N
\qquad&\Longrightarrow\qquad
\big|\xi^*(s)\big|\leq 2\cdot  r
\\
\noalign{\smallskip}
x\in[-36\cdot  r, 36\cdot  r]^N
\qquad&\Longrightarrow\qquad
\big|\xi^*(s)\big|\leq 72 \cdot  r
\end{aligned}
\qquad\quad\forall~s\in[0,t]\,;
\end{equation}
\item[--]
letting $\xi^*$ be any maximizer for
\begin{equation*}
\tag*{$\text{(CV)}^{t,x}$}
\max_{\xi\in AC([t,\tau],\mathbb{R}^N)}\bigg\lbrace{u_0(\xi(\tau))-
\int_t^{\tau}L(\xi(s),\dot{\xi}(s))ds\ \Big|\ \xi(t)=x\ \bigg\rbrace},
\end{equation*}
%
with $t\leq \tau$, one has 
\begin{equation}
\begin{aligned}
\xi^*(\tau)\in[- r,  r]^N
\qquad&\Longrightarrow\qquad
\big|\xi^*(s)\big|\leq 2\cdot  r
\\
\noalign{\smallskip}
\xi^*(\tau)\in[-6\cdot r, 6\cdot r]^N
\qquad&\Longrightarrow\qquad
\big|\xi^*(s)\big|\leq 12\cdot  r
\\
\noalign{\smallskip}
\xi^*(\tau)\in[-36\cdot  r, 36\cdot  r]^N
\qquad&\Longrightarrow\qquad
\big|\xi^*(s)\big|\leq 72 \cdot  r
\end{aligned}
\qquad\qquad\forall~s\in[t,\tau]\,.
\end{equation}
\end{itemize}
\item[b)] 
For every $\tau\leq \tau_2(6\cdot  r)$, $u_{0}\in \Lip(\mathbb{R}^N)\cap\mathbf{L}^\infty(\mathbb{R}^N)$
with $\|u_0\|_{\strut\mathbf{L}^\infty(\mathbb{R}^N)}\leq c_7(6\cdot  r)$,
the following holds:
\begin{itemize}
\item[--]
letting $\xi^*$ be any minimizer for~$(CV)_{t,x}$, $t\leq \tau $, $x\in\R^N\setminus \,]\!-6\cdot  r,\,6\cdot  r[^N$, one has 
\begin{equation}
\label{xistar-bounds4b}
\big|\xi^*(s)\big|> \frac{|x|}{3}
\qquad\quad\forall~s\in[0,t]\,.
\end{equation}
\item[--]
letting $\xi^*$ be any maximizer for~$(CV)^{t,x}$, $t\leq \tau $, 
with $\xi^*(\tau)\in\R^N\setminus \,]\!-6\cdot  r,\,6\cdot  r[^N$, one has 
\begin{equation}
\label{xistar-bounds4c}
\big|\xi^*(s)\big|> \frac{|\xi^*(\tau)|}{3}
\qquad\quad\forall~s\in[t,\tau]\,.
\end{equation}
\end{itemize}
\item[c)] 
Setting
\begin{equation}
\label{Irbar-def}
I_{r}\doteq\big [\!-\!72\cdot  r,\, 72\cdot  r\big],
\end{equation}
there holds
\begin{equation}
\label{linf-St0-bound-2}
\big|S_t 0(x)\big|\leq c_8(72 \cdot  r)\cdot t\qquad\quad
\forall~x\in I_{r}^N,
\ t >0\,. 
\end{equation} 
\item[d)] For every $\tau\leq \tau_3\big(c_6(36\cdot  r), 72\cdot  r,  m\big)$, 
and for any  $u_{0}\in \Lip(\mathbb{R}^N)\cap\mathbf{L}^\infty(\mathbb{R}^N)$
with \linebreak
$\|u_0\|_{\strut\mathbf{L}^\infty(\mathbb{R}^N)}\leq c_6(36\cdot  r)$,
 one has
\begin{equation}
\label{Lip-Stu0}
\begin{aligned}
\Lip[u_0]\leq  m
\qquad&\Longrightarrow\qquad
\Lip\big[S_\tau u_0;\, I_{r}^N\big]\leq 2\cdot m\,,
\\
\noalign{\smallskip}
\Lip[u_0]\leq 5 m
\qquad&\Longrightarrow\qquad
\Lip\big[S_\tau u_0;\, I_{r}^N\big]\leq 10\cdot m\,.
\end{aligned}
\end{equation}
\end{itemize}
Next, 
set
\begin{equation}
\label{tbar2-def}
\overline\tau(r,m)\doteq
\min
\bigg\{
\tau_1(36\!\cdot\! r),\, \tau_2(6\!\cdot\! r),\, \tau_3\big(c_6(36\!\cdot\! r), 72\!\cdot\! r, m\big),\, 
\frac{\min\{c_6(36\!\cdot\! r), c_7(6\!\cdot\! r)\}-8 m\!\cdot\!\sqrt{N} \!\cdot\!  r}{2 c_8(72\!\cdot\!  r)}
\bigg\}\!,
\end{equation}
%
and let $\overline k(r,m)>0$ be
a semiconcavity constant for $S_\tau 0$, $\tau\leq\overline\tau(r,m)$,
on $I_{r}^N$. 
%
%
%
Then, fix $K\geq \overline k(r,m)$, and take
\begin{equation}
\label{tbar3-def}
\tau\leq \overline\tau'(r,m)\doteq \min
\Big\{
\overline\tau(r,m),\, 
\tau_5(c_6(36\cdot r), 72\cdot  r, 3K)
\Big\},
\end{equation}
where $\tau_5(c_6(36\cdot r), 72\cdot  r, 3K)$ is
a constant with the property stated in Remark~\ref{semiconvexity}.
Observe that, by property~(d) above, one has $\Lip[S_\tau 0;\, I_{r}^N]\leq 2\cdot m$.
Moreover,  since the zero map is semiconvex  with semiconvexity
costant any $K>0$, by virtue of Remark~\ref{semiconvexity} we deduce that  $S_t 0(x)$ provides
a $C^1$ classical solution of~\eqref{HJ} on~$[0, \tau]\times I_{r}^N$.
This, in particular, implies by~\eqref{psit-def} that $S_\tau 0$ is a $C^1$ map on $I_{r}^N$.
Then, given 
\begin{equation}
\label{psi-final-data-1}
\psi \in \mathcal{SC}^{^{S_{\tau} 0, 1}}_{[r,\,4 m,\,2K]},
\end{equation}
%
we can define a $C^1$ semiconcave map $\widetilde\psi^\tau:\R^N\to\R$ with the properties:
\begin{itemize}
\item[(i)]
\begin{equation}
\label{psit-def}
\begin{aligned}
\widetilde\psi^\tau(x)&=\psi(x)\qquad\qquad\forall~x\in [- r,  r]^N\,,
\\
\widetilde\psi^\tau(x)&=S_\tau 0(x)\qquad\ \ \ \, \forall~x\in 
I_{r}^N\setminus [- r,  r]^N\,;
\end{aligned}
\end{equation}
\item[(ii)]
$\widetilde\psi^\tau$ has Lipschitz constant $5 m$ and  semiconcavity constant $3K$ on $\R^N$;
\item[(iii)]
\begin{equation}
\label{linf-bound-psit1}
\big\|\widetilde\psi^\tau\big\|_{\strut\mathbf{L}^\infty(\R^N)}
\leq 2 \max\Big\{\big\|\psi\big\|_{\strut\mathbf{L}^\infty([- r,  r]^N)},
\big\|S_{\tau}0\big\|_{\strut\mathbf{L}^\infty(I_{r}^N)}
\Big\}\,. 
\end{equation}
\end{itemize}
Recall that, by definitions~\eqref{LRM}, \eqref{LSRM}, \eqref{LSRM-1},
 $\psi$ is a Lipschitz continuous map on $[- r,  r]^N$
with Lipschitz constant~$4 m$, and that $\psi = S_{\tau} 0$, 
$D \psi = D S_{\tau} 0$
on $\partial [- r,  r]^N$.
Hence, because of~\eqref{linf-St0-bound-2}, \eqref{tbar2-def}, 
one finds
\begin{equation*}
\big\|S_{\tau} 0\big\|_{\strut\mathbf{L}^\infty(I_{r}^N)}
\leq c_8(72\cdot  r)\cdot \overline\tau \leq \frac{\min\{c_6(36\cdot r), c_7(6\cdot r)\}}{2},
\end{equation*}
and
\begin{equation*}
\begin{aligned}
\big\|\psi\big\|_{\strut\mathbf{L}^\infty([- r,  r]^N)}
&\leq 
\big\|S_{\tau} 0\big\|_{\strut\mathbf{L}^\infty([- r,  r]^N)} + 
4 m\cdot \sqrt{N} \cdot  r\,,
\\
\noalign{\smallskip}
&\leq c_8(72\cdot  r)\cdot \overline\tau + 4 m\cdot \sqrt{N} \cdot  r
\\
\noalign{\smallskip}
&\leq \frac{\min\{c_6(36\cdot r), c_7(6\cdot r)\}}{2},
\end{aligned}
\end{equation*}
which, in turn, together with \eqref{linf-bound-psit1}, yields
\begin{equation}
\label{linf-bound-psi2}
\big\|\widetilde\psi^\tau\big\|_{\strut\mathbf{L}^\infty(\R^N)}
\leq \min\big\{c_6(36\cdot  r), \, c_7(6\cdot  r)\big\}\,.
\end{equation}
\quad
\\
\\
{\bf 2.} 
We will show that, for 
$\tau$ satisfying~\eqref{tbar3-def}, with $\overline\tau(r,m)$ as in~\eqref{tbar2-def}, 
the map $\widetilde\psi^\tau$ defined above can be obtained as the value at time $\tau$
of a classical solution to~\eqref{HJ} by reversing the direction of time 
and constructing a backward solution to~\eqref{HJ} that starts at time $\tau$ from $\widetilde\psi^\tau$.
Namely, set
\begin{equation}
\label{in-data-back}
w_0^\tau(x)\doteq -\widetilde\psi^\tau(-x)
\qquad\forall~x\in\mathbb{R}^N,
\end{equation}
and consider the viscosity solution 
$w^\tau(t,x)$ of
\begin{equation}
\label{HJhat}
w_t(t,x)+H\big(\!-\!x,\nabla_{\!x}  w(t,x)\big)=0\,,
\qquad t\geq 0,\quad x\in \mathbb{R}^N,
\end{equation}
with initial datum
\begin{equation}
\label{forw-sol-def1}
w(0,\cdot)=w_0^\tau\,.
\end{equation}
Notice that the Hamiltonian $\widehat H (x,p)\doteq H(-x,p)$ satisfy the assumptions~{\bf (H1)-(H2)}
as \linebreak does~$H(x,p)$.
Moreover, by~\eqref{in-data-back} and because of (ii), $w_0^\tau$ is semiconvex with semiconvexity
costant~$3K$.
Thus, invoking Remark~\ref{semiconvexity}
and thanks to~\eqref{tbar3-def}, \eqref{linf-bound-psi2}, we know that 
the function 
$w(t,x)\doteq w^{\tau}(t,x)$
is a $C^1$~classical solution of~\eqref{HJ} on $[0, \tau]\times I_{r}^N$,
with $I_{r}$ as in~\eqref{Irbar-def}.
Furthermore, by properties (d) and (ii) above, and by virtue of~\eqref{tbar2-def}, \eqref{tbar3-def},  we deduce that 
\begin{equation}
\label{lipsch-w4}
\Lip\big[w(\tau,\cdot) ; I_{r}^N\big]\leq 10\, m\,.
\end{equation}
Next, observe that by a direct computation the function
\begin{equation}
\label{backw-sol-def1}
u(t,x)\doteq - w(\tau-t,-x)
\end{equation}
is also a $C^1$ classical solution of~\eqref{HJ} on $[0, \tau] \times I_{r}^N$.
Moreover, because of~\eqref{in-data-back}, \eqref{forw-sol-def1}, \eqref{backw-sol-def1}, one has
\begin{equation}
\label{final-forw-sol}
u(\tau,x)= - w_0^{\tau}(-x)= \widetilde\psi^{\tau}(x)
\qquad\forall~x\in\R^N.
\end{equation}
%
Then, for every 
$y\in  [-18\cdot  r, 18\cdot r]^N$,
consider
the pair $(\xi^y, p^y) \in  \big(AC([0,\tau],\mathbb{R}^N)\big)^{\!2}$
that  satisfies the Hamiltonian system 
\begin{equation}
\label{hamilt-syst-3}
\begin{cases}
\hspace{.cm}
\dot \xi = D_p H(\xi, p),&
\\
\hspace{.cm}
\dot p = - D_x H(\xi, p),&
\end{cases}
\end{equation}
on $[0,\tau]$, with terminal condition
\begin{equation}
\label{final-data-3}
\begin{cases}
\hspace{.cm}
\xi^y(\tau) =y,&
\\
\hspace{.cm}
p^y (\tau)= D \widetilde\psi^{\tau}(y).&
\end{cases}
\end{equation}
Notice that, 
for any $t\in [0,\tau]$, $y\in  [-36\cdot  r, 36\cdot  r]^N$,
the restriction of $\xi^y$ to $[t, \tau]$ provides the (unique) optimal solution 
for the backward maximization problem
\begin{equation*}
\tag*{$\text{(CV)}^{t,\xi^y(t)}$}
\begin{aligned}
&\max_{\xi\in AC([t,\tau],\mathbb{R}^N)}\bigg\lbrace{\widetilde\psi^{\tau}(\xi(\tau))-\int_t^{\tau}L(\xi(s),\dot{\xi}(s))ds\ \Big|\ \xi(t)=\xi^y(t)\ \bigg\rbrace}
\end{aligned}
\end{equation*}
(cfr.~\cite{BCJS}).
Moreover, since by~\eqref{tbar2-def}, \eqref{tbar3-def}
we have chosen $\tau\leq \min\{\tau_1(36\cdot  r), \, \tau_2(6\cdot  r)\}$,
and because of~\eqref{linf-bound-psi2},
relying on properties (a), (b) at point 1, and  recalling \eqref{Irbar-def} 
we find
\begin{equation}
\label{xistar-bounds4d}
\begin{aligned}
y\in  [- r,  r]^N
\qquad&\Longrightarrow\qquad
\xi^y(s)\in B\big(0, 2\cdot r\big)
%
\\
\noalign{\smallskip}
y\in  [-6\cdot  r,\, 6\cdot  r]^N
\qquad&\Longrightarrow\qquad
\xi^y(s)\in  B\big(0, 12\cdot r\big)
%
\\
\noalign{\smallskip}
y\in  \partial\, [-6\cdot  r,\, 6\cdot  r]^N
\qquad&\Longrightarrow\qquad
\xi^y(s)\in  \R^N\setminus B\big(0, 2\cdot r\big)
\\
\noalign{\smallskip}
y\in  [-36\cdot  r, 36\cdot  r]^N
\qquad&\Longrightarrow\qquad
\xi^y(s)\in  I_{r}
\end{aligned}
\quad\qquad\forall~s\in[0,\tau]\,.
\end{equation}
In particular, if $y\in  [-36\cdot  r, 36\cdot  r]^N\setminus \,]- r,  r[^N$,
because of~\eqref{psit-def} the pair $(\xi^y, p^y)$ satisfies the Hamiltonian system~\eqref{hamilt-syst-3}
with terminal condition
\begin{equation}
\label{final-data-4}
\begin{cases}
\hspace{.cm}
\xi^y(\tau) =y,&
\\
\hspace{.cm}
p^y (\tau)=  D S_{\tau} 0(y),&
\end{cases}
\end{equation}
and 
the restriction of $\xi^y$ to $[t, \tau]$ provides the (unique) optimal solution 
for the backward maximization problem
\begin{equation*}
\tag*{$\text{(CV)}^{t,\xi^y(t)}$}
\max_{\xi\in AC([t,\tau],\mathbb{R}^N)}\bigg\lbrace{S_{\tau}0(\xi(\tau))-\int_t^{\tau}L(\xi(s),\dot{\xi}(s))ds\ \Big|\ \xi(t)=\xi^y(t)\ \bigg\rbrace}.
\end{equation*}
Hence, since $u(t,x)$ and $S_t 0(x)$ are both $C^1$~classical solutions of~\eqref{HJ} 
on $[0, \tau]\times I_{r}^N$, we deduce
\begin{equation}
\label{u-back-max}
\begin{aligned}
u(t, \xi^y(t)) &= 
\widetilde\psi^{\tau}(y)-\int_t^{\tau}L(\xi^y(s),\dot{\xi^y}(s))ds
\\
&=S_{\tau}0(y)-\int_t^{\tau}L(\xi^y(s),\dot{\xi^y}(s))ds
\\
&=S_t 0 (\xi^y(t)),
\end{aligned}
\end{equation}
for every $ t\in [0,\tau], \ y\in [-36\cdot  r, 36\cdot  r]^N\setminus \,]- r,  r[^N$.
This, in particular, implies that 
\begin{equation}
\label{u-back-max-indata}
\qquad\qquad\qquad u(0, \xi^y(0)) = S_0\, 0( \xi^y(0))=0
\qquad\ \forall~y\in [-36\cdot  r,\, 36\cdot  r]^N\setminus \,]- r,  r[^N.
\end{equation}
Moreover,  for every $y\in [-36\cdot  r,\, 36\cdot  r]^N$,
the first component $\xi^y$ of the solution to~\eqref{hamilt-syst-3}-\eqref{final-data-3}
provides the minimizer for~$(CV)_{\tau,y}$,  with initial cost $u(0,\cdot)$. 
\quad
\\
{\bf 3.} 
Fix $\tau>0$ satisfying~\eqref{tbar2-def}, \eqref{tbar3-def}
and let $u(t,x)$ be the map defined in~\eqref{backw-sol-def1}. For every 
$y\in  [-6\cdot  r, 6\cdot r]^N$,
let $\xi^y$ be the first component of the solution to~\eqref{hamilt-syst-3}-\eqref{final-data-3}.
By the regularity of $\widetilde\psi^{\tau}$ it follows that the sets
\begin{equation}
\label{D0-def}
\qquad
\Lambda_1\doteq
\Big\{\xi^y(0) \ \big| \  y\in  \partial\,[-  r,  r]^N
\Big\},
\qquad\quad
\Lambda_2\doteq
\Big\{\xi^y(0) \ \big| \  y\in  \partial\,[-6\cdot   r, 6\cdot   r]^N
\Big\}
\end{equation}
are piecewise $C^1$, closed, hypersurfaces that separate $\R^N$ in two connected
components. Call $\Omega_i$, $i=1,2$, the bounded connected domains that
have $\Lambda_i$ as boundary, so that there holds
%
\begin{equation}
\label{Omt-def1}
\begin{aligned}
\Big\{\xi^y(0) \ \big| \  y\in  \,]\!- r,\,  r[^N
\Big\}&\subset \Omega_1,
\qquad\quad
\partial\, \Omega_1= \Lambda_1\,,
\\
\noalign{\smallskip}
\Big\{\xi^y(0) \ \big| \  y\in  \,]\!-6\cdot  r, \,6\cdot  r[^N
\Big\}&\subset \Omega_2,
\qquad\quad
\partial\, \Omega_2= \Lambda_2\,.
\end{aligned}
\end{equation}
Observe that, by~\eqref{Irbar-def}, \eqref{xistar-bounds4d}, 
and by the definition of $\Omega_i$, we have
\begin{equation}
\label{xistar-bounds5}
\overline{\Omega}_1\subset B\big(0, 2\cdot r\big)
\subset\overline{\Omega}_2 \subset  B\big(0, 12\cdot r\big)
\subset I_{r}^N.
\end{equation}
Moreover, 
because of~\eqref{u-back-max-indata}, one has
\begin{equation}
\label{u0-Omr}
u(0, x)=0\qquad\quad\forall~x\in  \Lambda_1\cup \Lambda_2\,.
\end{equation}
Hence, by virtue of~\eqref{xistar-bounds5},
and recalling~\eqref{lipsch-w4}, \eqref{backw-sol-def1}, we deduce
\begin{equation}
\label{linf-bound-u0}
|u(0,x)|\leq 120\, m\, \cdot r
\qquad\quad\forall~x\in {\Omega}_2\,.
\end{equation}
Then, define the function
\begin{equation}
\label{u0-sharp}
u_0^\sharp (x)\doteq
\begin{cases}
u(0, x)\quad&\text{if}\qquad x\in \Omega_1
\\
\max\big\{0, u(0, x)\big\}  \quad&\text{if}\qquad x\in  \Omega_2\setminus \Omega_1
%
\\
\ \  0 \quad&\text{if}\qquad x\in \R^N\setminus \Omega_2
\end{cases}
\end{equation}
and notice that, by~\eqref{xistar-bounds5}, \eqref{u0-Omr}, $u_0^\sharp$ is a continuous map, while
\eqref{lipsch-w4}, \eqref{backw-sol-def1}, \eqref{xistar-bounds5}, \eqref{u0-sharp}
imply
\begin{equation}
\label{u0-sharp-prop1}
\mathrm{supp}\big(u_0^\sharp\big)\subset [-12\cdot  r, 12\cdot  r]^N\,,
\qquad\qquad
\Lip\big[u_0^\sharp\big]\leq 10\, m\,.
\end{equation}
Therefore, recalling definition~\eqref{CLM-1}, and because of~\eqref{rbar-mbar-conds1-2}, there holds
\begin{equation}
\label{u0-sharp-prop2}
u_0^\sharp\in \mathcal L_{[12\cdot  r,\, 10\, m]}
\subset \mathcal L_{[R,M]}\,.
\end{equation}
We claim that
\begin{equation}
\label{u0-sharp-prop3}
S_{\tau} u_0^\sharp (x) =
\begin{cases}
\psi (x)\quad&\text{if}\qquad x\in [- r,  r]^N,
\\
\noalign{\smallskip}
S_{\tau} 0 (x) \quad&\text{if}\qquad x\in \R^N\setminus [- r,  r]^N.
\end{cases}
\end{equation}
In fact, for every $x\in\R^N$ let $\xi^\sharp_x$, $\xi^*_x$ be any minimizer for~$(CV)_{\tau,x}$ 
with initial cost~$u_0^\sharp$, $u(0,\cdot)$, respectively, so that one has
\begin{equation}
\label{Stauub-1}
\begin{aligned}
\xi^\sharp_x(\tau)&=x,\qquad\quad 
S_{\tau} u_0^\sharp (x)= u_0^\sharp(\xi^\sharp_x(0))+\int_0^\tau L(\xi^\sharp_x(s),\dot{\xi^\sharp_x}(s))ds\,,
\\
\noalign{\smallskip}
\xi^*_x(\tau)&=x,\qquad\quad 
u(\tau,x)= u(0,\xi^*_x(0))+\int_0^\tau L(\xi^*_x(s),\dot{\xi^*_x}(s))ds\,.
\end{aligned}
\end{equation}
Observe that, because of~\eqref{rbar-mbar-conds1-2},  \eqref{linf-bound-u0}, \eqref{u0-sharp}, one has
\begin{equation}
\label{linf-bound-u0-2}
\|u_0^\sharp\|_{\strut\mathbf{L}^\infty(\mathbb{R}^N)}\leq 
\min\big\{c_6(36\cdot  r), \, c_7(6\cdot  r)\big\}.
\end{equation}
Then, by the choice of $\tau$ in~\eqref{tbar2-def}, \eqref{tbar3-def},
relying on properties (a), (b) at point 1 and on~\eqref{xistar-bounds5}, we deduce that 
\begin{equation}
\label{xistar-bounds6}
\begin{aligned}
x \in[- r,  r]^N
\qquad&\Longrightarrow\qquad
\xi^\sharp_x(0)\in \Omega_2
\\
\noalign{\smallskip}
x\in \R^N\setminus [-36\cdot  r, \, 36\cdot  r]^N
\qquad&\Longrightarrow\qquad
\xi^\sharp_x(0)\in  \R^N\setminus \Omega_2.
\end{aligned}
\end{equation}
\quad
\\
{\bf 4.} 
In order to establish~\eqref{u0-sharp-prop3}, we shall distinguish three cases.
\\
\\
\noindent
Case 1: $x\in [- r, \,  r]^N$.\\
 By definition of $\Omega_1$ we have $\xi^*_x(0)\in\Omega_1$
for all $x\in [- r, \,  r]^n$.
Therefore, because of~\eqref{u0-sharp},
there holds $u(0, \xi^*_x(0)) = u^\sharp(\xi^*_x(0))$, while~\eqref{u0-sharp}, \eqref{xistar-bounds6}
imply $u^\sharp(\xi^\sharp_x(0))\geq u(0, \xi^\sharp_x(0))$.
Hence,  by~\eqref{Stauub-1}, we deduce
\begin{equation}
\label{St-eq1}
\begin{aligned}
u(\tau,x)&= u^\sharp(\xi^*_x(0)) +\int_0^\tau L(\xi^*_x(s),\dot{\xi^*_x}(s))ds
\\
\noalign{\smallskip}
&\geq S_{\tau} u_0^\sharp (x)
\\
\noalign{\smallskip}
&= u_0^\sharp(\xi^\sharp_x(0))+\int_0^\tau L(\xi^\sharp_x(s),\dot{\xi^\sharp_x}(s))ds
\\
\noalign{\smallskip}
&\geq u(0, \xi^\sharp_x(0)) + \int_0^\tau L(\xi^\sharp_x(s),\dot{\xi^\sharp_x}(s))ds
\\
\noalign{\smallskip}
&\geq u(\tau,x),
\end{aligned}
\end{equation}
which proves~\eqref{u0-sharp-prop3},
recalling \eqref{psit-def}, \eqref{in-data-back}, \eqref{forw-sol-def1}, \eqref{backw-sol-def1}.
\\
\\
\noindent
Case 2: $x\in [-36\cdot  r, \, 36\cdot  r]^N\setminus [- r, \, r]^N$.\\
By the observations at point 3 and because of~\eqref{u-back-max-indata}, \eqref{u0-sharp},
we know that
$u(0, \xi^*_x(0)) = 0 = u^\sharp(\xi^*_x(0))$
for all $x\in [-36\cdot  r, \, 36\cdot  r]^N\setminus [- r, \,  r]^N$.
Moreover, by~\eqref{Irbar-def},
\eqref{psit-def}, \eqref{in-data-back}, \eqref{forw-sol-def1}, \eqref{backw-sol-def1}, one has
\begin{equation}
\label{u0-St0}
u(\tau,x)=S_\tau 0(x)\qquad\quad\forall~x\in [-36\cdot  r, \, 36\cdot  r]^N\setminus [- r, \,  r]^N\,.
\end{equation}
Thus, if $u_0^\sharp(\xi^\sharp_x(0))\geq u(0, \xi^\sharp_x(0))$, by~\eqref{Stauub-1} we derive
\begin{equation}
\label{St-eq2}
\begin{aligned}
u(\tau,x)&= \int_0^\tau L(\xi^*_x(s),\dot{\xi^*_x}(s))ds
\\
\noalign{\smallskip}
&\geq S_{\tau} u_0^\sharp (x)
\\
\noalign{\smallskip}
&= u_0^\sharp(\xi^\sharp_x(0))+\int_0^\tau L(\xi^\sharp_x(s),\dot{\xi^\sharp_x}(s))ds
\\
\noalign{\smallskip}
&\geq u(0, \xi^\sharp_x(0)) + \int_0^\tau L(\xi^\sharp_x(s),\dot{\xi^\sharp_x}(s))ds
\\
\noalign{\smallskip}
&\geq u(\tau,x).
\end{aligned}
\end{equation}
Otherwise, if $u_0^\sharp(\xi^\sharp_x(0))< u(0, \xi^\sharp_x(0))$,
by\eqref{u0-sharp}
it must be $u_0^\sharp(\xi^\sharp_x(0))=0$.
Hence, relying on~\eqref{Stauub-1}, \eqref{u0-St0}, we get
\begin{equation}
\label{St-eq3}
\begin{aligned}
S_\tau 0(x)&= \int_0^\tau L(\xi^*_x(s),\dot{\xi^*_x}(s))ds
\\
\noalign{\smallskip}
&\geq S_{\tau} u_0^\sharp (x)
\\
\noalign{\smallskip}
&= \int_0^\tau L(\xi^\sharp_x(s),\dot{\xi^\sharp_x}(s))ds
\\
\noalign{\smallskip}
&\geq S_\tau 0(x).
\end{aligned}
\end{equation}
Together~\eqref{u0-St0}, \eqref{St-eq2}, \eqref{St-eq3}, yield~\eqref{u0-sharp-prop3}.
\\
\\
\noindent
Case 3: $x\in \R^N\setminus [-36\cdot  r, \, 36\cdot  r]^N$.\\
By~\eqref{u0-sharp}, \eqref{xistar-bounds6} we know that 
$u_0^\sharp(\xi^\sharp_x(0))=0$ for all $x\in \R^N\setminus [-36\cdot  r, \, 36\cdot  r]^N$.
Moreover, letting $\xi_x^o$ be  be any minimizer for~$(CV)_{\tau,x}$ 
with initial cost zero, 
relying on properties (b) at point 1 and on~\eqref{xistar-bounds5}, \eqref{u0-sharp},
we deduce that 
also $u_0^\sharp(\xi^o_x(0))=0$ for all $x\in \R^N\setminus [-36\cdot  r, \, 36\cdot  r]^N$.
Then, using~\eqref{Stauub-1}, we derive
\begin{equation}
\label{St-eq4}
\begin{aligned}
S_\tau 0(x)&= \int_0^\tau L(\xi^o_x(s),\dot{\xi^o_x}(s))ds
\\
\noalign{\smallskip}
&\geq S_{\tau} u_0^\sharp (x)
\\
\noalign{\smallskip}
&= \int_0^\tau L(\xi^\sharp_x(s),\dot{\xi^\sharp_x}(s))ds
\\
\noalign{\smallskip}
&\geq S_\tau 0(x),
\end{aligned}
\end{equation}
which proves~\eqref{u0-sharp-prop3}. \\
\\
This completes the proof of the proposition
taking  
\begin{equation}
\label{r1-m1-def-2}
\begin{gathered}
r_1(R)\doteq \frac{R}{12},
\\
\noalign{\smallskip}
m_1(r,M)\doteq 2
 \min
 \bigg\{\frac{M}{10},\
\frac{1}{120\cdot \sqrt{N}\cdot  r}\cdot \min\Big\{c_6(36\cdot  r), \, c_7(6\cdot  r)\Big\}
\bigg\},
\\
\noalign{\bigskip}
\tau_6(r,m)\doteq \overline\tau'(r,m/2),
\end{gathered}
\end{equation}
with $\overline\tau'$ as in~\eqref{tbar2-def}, \eqref{tbar3-def}, 
and letting ${k_1(r,m)}$ be
a semiconcavity constant for $S_\tau 0$, $\tau\leq \tau_6(r,m)$
on $I_{r}^N$.
\end{proof}
\smallskip
\begin{remark}
\label{control-prop1}
\rm
By Proposition~\ref{control-part}, for every $r\leq r_1(R)$, $m\leq m_1(r,M)$,
$T\leq \tau_6(r,m)$ and $K=k_1(r,m)$, one has
\begin{equation}
\label{incl-sc}
\widetilde{\mathcal{SC}}^{^{S_T 0, 1}}_{[r,\,2m,\,2K]}\subset S_T(\mathcal L_{[R,M]})\,,
\end{equation}
%
%
where $\mathcal L_{[R,M]}$ is the set  defined  in \eqref{CLM-1} while
\begin{equation}
\label{LSRM-3}
\widetilde{\mathcal{SC}}^{^{S_T 0, 1}}_{[r,\,2m,\,2K]}\doteq 
\Big\lbrace{
u_0\in \mathcal L_{[r,2m]}\, \big| \, \ u_0\!\!\restriction_{[-r,r]^N}\in
\mathcal{SC}^{^{S_T 0, 1}}_{[r,\,2m,\,2K]}, \ \  \ u_0\!\!\restriction_{R^N\setminus [-r,r]^N} = S_T 0
\Big\rbrace},
\end{equation}
with $\mathcal{SC}^{^{S_T 0, 1}}_{[r,\,m,\,K]}$  defined 
as in \eqref{LSRM-1}.
\par \medskip
\noindent
It remains an interesting open problem to analyze the global in time exact controllability of~\eqref{HJ}.
Namely, one would like to determine wether there exist constants $r_1(R,M), m_1(r,M)$ so that, for every time $T>0$
and for any $r\leq r_1(R,M)$, $m\leq m_1(r,M)$, 
there holds~\eqref{incl-sc} for some $K=k_1(r,m,T)$.
\end{remark}

\medskip

\subsection{Conclusion of the proof of Theorem~\ref{upper-lower-estimate}-${\bf (ii)}$}
\label{concl-proof-mainthm}

Given $R,M>0$, let  $r_1(R), m_1(r,M), k_1(r,m)$ and
$\tau_6(r,m)$ be the constants provided by Proposition~\ref{control-part}
and set $r_{\!_R}\doteq r_1(R)$, $m_{\!_{R,M}}\doteq m_1(r_{\!_R},M)$, $K_{\!_{R,M}}\doteq k_1(r_{\!_R},m_{\!_{R,M}})$,
$\tau_{\!_{R,M}}\doteq \tau_6(r_{\!_R},m_{\!_{R,M}})$,
so that, as observed in Remark~\ref{control-prop1}, there holds
\begin{equation}
\label{incl-sc-2}
\widetilde{\mathcal{SC}}^{^{S_T 0, 1}}_{[r_{\!_R},\,2m_{\!_{R,M}},\,2K_{\!_{R,M}}]}\subset 
S_T(\mathcal L_{[R,M]})
\qquad\quad\forall~T\leq \tau_{\!_{R,M}}\,.
\end{equation}
On the other hand, by definition~\eqref{LSRM-3}, one has
\begin{equation}
\label{supp-LSRM-3}
\text{supp}(\varphi)\subset [-r_{\!_R}, r_{\!_R}]^N\qquad\quad \forall~\varphi\in 
\widetilde{\mathcal{SC}}^{^{S_T 0, 1}}_{[r_{\!_R},\,2m_{\!_{R,M}},\,2K_{\!_{R,M}}]}-S_T \,0\,.
\end{equation}
Therefore, relying on~\eqref{incl-sc-2}, \eqref{supp-LSRM-3},
and applying Theorem~\ref{estSC}-(ii), we deduce that, for $T\leq \tau_{\!_{R,M}}$ and
$\varepsilon$ sufficiently small, there holds
\begin{equation}
\label{lower-entr-est-final}
\begin{aligned}
\mathcal{H}_{\epsilon}\Big(S_T(\mathcal{L}_{[R,M]})-S_T \,0\ \big|\ \mathbf{W}^{1,1}(\mathbb{R}^N)\Big)
&\geq
\mathcal{H}_{\epsilon}\Big(\widetilde{\mathcal{SC}}^{^{S_T 0, 1}}_{[r_{\!_R},\,2m_{\!_{R,M}},\,2K_{\!_{R,M}}]}-S_T \,0\ \big|\ \mathbf{W}^{1,1}(\mathbb{R}^N)\Big)
\\
\noalign{\smallskip}
&=
\mathcal{H}_{\epsilon}\Big(\widetilde{\mathcal{SC}}^{^{S_T 0, 1}}_{[r_{\!_R},\,2m_{\!_{R,M}},\,2K_{\!_{R,M}}]}-S_T \,0 \ \big|\ \mathbf{W}^{1,1}([-r_{\!_R}, r_{\!_R}]^N)\Big)
\\
\noalign{\smallskip}
&=\mathcal{H}_{\epsilon}\Big(\widetilde{\mathcal{SC}}^{^{S_T 0, 1}}_{[r_{\!_R},\,2m_{\!_{R,M}},\,2K_{\!_{R,M}}]} \ \big|\ \mathbf{W}^{1,1}([-r_{\!_R}, r_{\!_R}]^N)\Big)
\\
\noalign{\smallskip}
&\geqslant \gamma^-_{_{[l_T,\mu_T,\kappa_T,N]}}\cdot\frac{1}{\varepsilon^N}\,,
\end{aligned}
\end{equation}
where
\begin{equation}
\gamma^-_{_{[r_{\!_R},K_{\!_{R,M}},N]}}=
\frac{1}{8\cdot\ln 2}\cdot \bigg(\frac{K_{\!_{R,M}}\, \omega_N\, r_{\!\!_R}^{N+1}}{48(N+1)\,2^{N+1}}\bigg)^{\!\!N}.
\end{equation}
This completes the proof of the lower bound~\eqref{Lower-est-H}.

\qed
\section*{Acknowledgements}
This work was partially supported by the National Group for Mathematical Analysis and Probability (GNAMPA) of the Istituto Nazionale di Alta Matematica ``Francesco Severi'' (INdAM) and by the  INdAM-CNRS European Research Group (GDRE) on Control of Partial Differential Equations (CONEDP).
Fabio Ancona was partially supported by the Miur-Prin 2012 Project "Nonlinear Hyperbolic Partial Differential Equations, 
Dispersive and Transport Equations: theoretical and applicative aspects" 
and by the University of Padova grant "PRAT 2013 - Traffic Flow on Networks: Analysis and Control".
\bigskip

\end{document}